\documentclass[english]{article}
\usepackage[T1]{fontenc}
\usepackage[latin1]{inputenc}
\usepackage{amsthm}
\usepackage{amsmath}
\usepackage{amssymb}
\usepackage[all]{xy}

\makeatletter
\theoremstyle{plain}
\newtheorem{thm}{Theorem}[section]
  \theoremstyle{remark}
  \newtheorem*{acknowledgement*}{Acknowledgement}
  \theoremstyle{plain}
  \newtheorem{lem}[thm]{Lemma}
  \theoremstyle{plain}
  \newtheorem{prop}[thm]{Proposition}
  \theoremstyle{plain}
  \newtheorem{cor}[thm]{Corollary}
  \theoremstyle{remark}
  \newtheorem{rem}[thm]{Remark}
  \theoremstyle{remark}
  \newtheorem{note}[thm]{Note}

\makeatother

\usepackage{babel, fullpage} 

\usepackage{etoolbox}
\preto{\chapter}{}%
\preto{\section}{}%
\preto{\subsection}{}%

\makeatletter
\@addtoreset{thm}{chapter}
\@addtoreset{thm}{section}
\@addtoreset{thm}{subsection}
\makeatother
\date{}

\makeatother

\usepackage{babel}

\begin{document}

\title{Monodromy and local-global compatibility for $l=p$}

\author{Ana Caraiani}
\maketitle
\begin{abstract}
We strengthen the compatibility between local and global Langlands
correspondences for $GL_{n}$ when $n$ is even and $l=p$. Let $L$
be a CM field and $\Pi$ a cuspidal automorphic representation of
$GL_{n}(\mathbb{A}_{L})$ which is conjugate self-dual and regular
algebraic. In this case, there is an $l$-adic Galois representation
associated to $\Pi$, which is known to be compatible with local Langlands
in almost all cases when $l=p$ by recent work of Barnet-Lamb, Gee,
Geraghty and Taylor. The compatibility was proved only up to semisimplification
unless $\Pi$ has Shin-regular weight. We extend the compatibility
to Frobenius semisimplification in all cases by identifying the monodromy
operator on the global side. To achieve this, we derive a generalization
of Mokrane's weight spectral sequence for log crystalline cohomology. 
\end{abstract}

\section{Introduction}

This paper is a continuation of \cite{C}. Here we extend our local-global
compatibility result to the case $l=p$. 
\begin{thm}
\label{main theorem}Let $n\in\mathbb{Z}_{\geq2}$ be an integer and
$L$ be a CM field with complex conjugation $c$. Let $l$ be a prime
of $\mathbb{Q}$ and $\iota_{l}:\bar{\mathbb{Q}}_{l}\to\mathbb{C}$
be an isomorphism. Let $\Pi$ be a cuspidal automorphic representation
of $GL_{n}(\mathbb{A}_{L})$ satisfying
\begin{itemize}
\item $\Pi^{\vee}\simeq\Pi\circ c$
\item $\Pi$ is cohomological for some irreducible algebraic representation
$\Xi$ of $GL_{n}(L\otimes_{\mathbb{Q}}\mathbb{C})$
\end{itemize}
Let \[
R_{l}(\Pi):Gal(\bar{L}/L)\to GL_{n}(\mathbb{\bar{Q}}_{l})\]
be the Galois representation associated to $\Pi$ by \cite{Sh,CH}.
Let $y$ be a place of $L$ above $l$. Then we have the following
isomorphism of Weil-Deligne representations \[
WD(R_{l}(\Pi)|_{Gal(\bar{L}_{y}/L_{y})})^{F-ss}\simeq\iota_{l}^{-1}\mathcal{L}_{n,L_{y}}(\Pi_{y}).\]

\end{thm}
Here $\mathcal{L}_{n,L_{y}}(\Pi_{y})$ is the image of $\Pi_{y}$
under the local Langlands correspondence, using the geometric normalization;
$WD(r)$ is the Weil-Deligne representation attached to a de Rham
$l$-adic representation $r$ of the absolute Galois group of an $l$-adic
field; $F-ss$ denotes Frobenius semisimplification. 

This theorem is proved in \cite{BLGGT1,BLGGT2} in the case when $\Pi$
has Shin-regular weight (either $n$ is odd or if $n$ is even then
$\Pi$ satisfies an additional regularity condition) and in general
up to semisimplification. Our goal is to match up the monodromy operators
in the case when $n$ is even and $\Pi$ does not necessarily have
Shin-regular weight. By Theorem 1.2 of \cite{C}, $\Pi_{y}$ is tempered,
so $\iota_{l}^{-1}\mathcal{L}_{n,L_{y}}(\Pi_{y})$ is pure (in the
sense of \cite{TY}) of some weight. By Lemma 1.4 (4) of \cite{TY},
given a semisimple representation of the Weil group of some $l$-adic
field, there is at most one way to choose the monodromy operator such
that the resulting Weil-Deligne representation is pure. 

By Theorem A of \cite{BLGGT2}, we already have an isomorphism up
to semisimplification. We note that Theorem A of \cite{BLGGT2} is
stated for an imaginary CM field $F$. For our CM field $L$ we proceed
as on pages 230-231 of \cite{HT} to find a quadratic extension $F/L$
which is an imaginary $CM$ field, in which $y=y'y''$ splits and
such that \[
[R_{l}(\Pi)|_{\mathrm{Gal}(\bar{L}/F)}]=[R_{l}(BC_{F/L}(\Pi))].\]
This together with Theorem A of \cite{BLGGT2} gives the compatibility
up to semisimplification for the place $y$ of $L$. Therefore, in
order to complete the proof of Theorem \ref{main theorem}, it suffices
to show that $W:=WD(R_{l}(\Pi)_{Gal(\bar{L}_{y}/L_{y})})^{F-ss}$
is pure of some weight when $n$ is even. From now on we will let
$n\in\mathbb{Z}_{\geq2}$ be an even integer. 

Our argument will follow the same general lines as that of \cite{TY}.
Our strategy involves reducing the problem to the case when $\Pi_{y}$
has an Iwahori fixed vector, finding in this case the tensor square
of $W$ in the log crystalline cohomology of a compact Shimura variety
with Iwahori level structure and finally computing a part of this
cohomology explicitly. For the last step, we need to derive a formula
for the log crystalline cohomology of the special fiber of the Shimura
variety in terms of the crystalline cohomology of closed Newton polygon
strata in the special fiber. Deriving this formula constitutes the
heart of this paper; we obtain it in the form of a generalization
of the Mokrane spectral sequence or as a crystalline analogue of Corollary
4.28 of \cite{C}. 

We briefly outline the structure of our paper. In Section 2 we reduce
to the case where $\Pi$ has an Iwahori fixed vector, we define an
inverse system of compact Shimura varieties associated to a unitary
group and show that the crystalline cohomology of the Iwahori-level
Shimura variety realizes the tensor square of $W$. The Shimura varieties
we work with are the same as those studied in \cite{C}, so in Section
2 we also recall the main results from \cite{C} concerning them.
In Section 3 we recall and adapt to our situation some standard results
from the theory of log crystalline cohomology and the de Rham-Witt
complex; we define and study some slight generalizations of the logarithmic
de Rham-Witt complex. In Section 4 we generalize the Mokrane spectral
sequence to our geometric setting. In Section 5 we prove Theorem \ref{main theorem}. 
\begin{acknowledgement*}
I am very grateful to my advisor, Richard Taylor, for suggesting this
problem and for his constant encouragement and advice. I would like
to thank Kazuya Kato, Mark Kisin, Jay Pottharst and Claus Sorensen
for many useful conversations and comments. I would also like to thank
Luc Illusie for his detailed comments on an earlier draft of this
paper and in particular for pointing me to Nakkajima's paper \cite{Na}. 
\end{acknowledgement*}

\section{Shimura varieties}

In this section we show that we can understand the Weil-Deligne representation
$W=WD(R_{l}(\Pi)_{Gal(\bar{L}_{y}/L_{u})})^{F-ss}$ by computing a
part of the crystalline cohomology of an inverse system of Shimura
varieties. In the first part we closely follow Sections 2 and 7 of
\cite{C} and afterwards we use some results from Section 5 of op.
cit.

We claim first that we can find a CM field extension $F'$ of $L$
such that
\begin{itemize}
\item $F'=EF_{1}$, where $E$ is an imaginary quadratic field in which
$l$ splits and $F_{1}=(F')^{c=1}$ has $[F_{1}:\mathbb{Q}]\geq2$,
\item $F'$ is soluble and Galois over $L$,
\item $\Pi_{F'}^{0}:=BC_{F'/L}(\Pi)$ is a cuspidal automorphic representation
of $GL_{n}(\mathbb{A}_{F'})$, and
\item there is a place $\mathfrak{p}$ above the place $y$ of $L$ such
that $\Pi_{F',\mathfrak{p}}^{0}$ has a nonzero Iwahori fixed vector
\end{itemize}
and a CM field $F$ which is a quadratic extension of $F'$ such that
\begin{itemize}
\item $\mathfrak{p}=\mathfrak{p}_{1}\mathfrak{p}_{2}$ splits in $F$,
\item $\mbox{Ram}_{F/\mathbb{Q}}\cup\mbox{Ram}_{\mathbb{Q}}(\Pi)\subset\mbox{Spl}_{F/F_{2},\mathbb{Q}}$,
where $F_{2}:=(F)^{c=1}$, and
\item $\Pi_{F}^{0}=BC_{F/F'}(\Pi_{F'}^{0})$ is a cuspidal automorphic representation
of $GL_{n}(\mathbb{A}_{F})$. 
\end{itemize}
We can find $F$ and $F'$ as in the proof of Corollary 5.9 of \cite{C}.
Since purity is preserved under finite extensions by Lemma 1.4 of
\cite{TY}, to show that $W$ is pure it suffices to show that \[
W_{F'}:=WD(R_{l}(\Pi_{F'}^{0})|_{Gal(\bar{F'_{\mathfrak{p}}}_{\mathfrak{}}/F'_{\mathfrak{p}})})^{F-ss}\]
is pure. Note that in this new situation $\Pi_{F',\mathfrak{p}}^{0}$
has a non-zero Iwahori-fixed vector. 

We can define an algebraic group $G$ over $\mathbb{Q}$ and an inverse
system of Shimura varieties over $F'$ corresponding to a PEL Shimura
datum $(F,*,V,\langle\cdot,\cdot\rangle,h)$. Here $ $$F$ is the
CM field defined above and $*=c$ is the involution corresponding
to complex conjugation. We take $V$ to be the $F$-vector space $F^{n}$.
The pairing \[
\langle\cdot,\cdot\rangle:V\times V\to\mathbb{Q}\]
is a non-degenerate Hermitian pairing such that $\langle fv_{1},v_{2}\rangle=\langle v_{1},f^{*}v_{2}\rangle$
for all $f\in F$ and $v_{1},v_{2}\in V$. The last element we need
is an $\mathbb{R}$-algebra homorphism $h:\mathbb{C}\to\mbox{End}_{F}(V)\otimes_{\mathbb{Q}}\mathbb{R}$
such that the bilinear pairing \[
(v_{1},v_{2})\to\langle v_{1},h(i)v_{2}\rangle\]
is symmetric and positive definite. We define the algebraic group
$G$ over $\mathbb{Q}$ \[
G(R)=\{(g,\lambda)\in\mbox{End}_{F\otimes_{\mathbb{Q}}\mathbb{R}}(V\otimes_{\mathbb{Q}}R)^{\times}\times R^{\times}|\langle gv_{1},gv_{2}\rangle=\lambda\langle v_{1},v_{2}\rangle\}\]
for any $\mathbb{Q}$-algebra $R$. 

We choose embeddings $\tau_{i}:F\hookrightarrow\mathbb{C}$ such that
$\tau_{2}=\tau_{1}\circ\sigma,$ where $\sigma$ is element of $\mbox{Gal}(F/F')$
which takes $\mathfrak{p}_{1}$ to $\mathfrak{p}_{2}$. For $\sigma\in\mbox{Hom}_{E,\tau_{E}}(F,\mathbb{C})$
we let $(p_{\sigma},q_{\sigma})$ be the signature at $\sigma$ of
the pairing $\langle\cdot,\cdot\rangle$ on $V\otimes_{\mathbb{Q}}\mathbb{R}.$
In particular, $\tau_{E}:=\tau_{1}|_{E}=\tau_{2}|_{E}$ is well-defined.
We claim that it is possible to choose a PEL datum as above such that
$(p_{\tau},q_{\tau})=(1,n-1)$ for $\tau=\tau_{1}$ or $\tau_{2}$
and $(p_{\tau},q_{\tau})=(0,n)$ otherwise and such that $G_{\mathbb{Q}_{v}}$
is quasi-split at every finite place $v$ of $\mathbb{Q}$. This follows
from Lemma 2.1 of \cite{C} and the discussion following it and it
depends crucially on the fact that $n$ is even. We choose such a
PEL datum and we let $G$ be the corresponding algebraic group over
$\mathbb{Q}$ with the prescribed signature at infinity and quasi-split
at all the finite places. 

Let $\Xi_{F}^{0}:=BC_{F/L}(\Xi)$ and $F_{2}=F{}^{c=1}$. The following
lemma is the same as Lemma 7.2 of \cite{Sh}.
\begin{lem}
Let $\Pi_{F}^{0}$ and $\Xi_{F}^{0}$ be as above. We can find a character
$\psi:\mathbb{A}_{E}^{\times}/E^{\times}\to\mathbb{C}^{\times}$ and
an algebraic representation $\xi_{\mathbb{C}}$ of $G$ over $\mathbb{C}$
satisfying the following conditions:
\begin{itemize}
\item $\psi_{\Pi_{F}^{0}}=\psi^{c}/\psi$
\item $\Xi_{F}^{0}$ is isomorphic to the restriction of $\Xi'$ to $R_{F/\mathbb{Q}}(GL_{n})\times_{\mathbb{Q}}\mathbb{C}$,
where $\Xi'$ is obtained from $\xi_{\mathbb{C}}$ by base change
from $G$ to $\mathbb{G}_{n}:=R_{E/\mathbb{Q}}(G\times_{\mathbb{Q}}E)$
\item $\xi_{\mathbb{C}}|_{E_{\infty}^{\times}}^{-1}=\psi_{\infty}^{c},$
and 
\item $\mbox{Ram}_{\mathbb{Q}}(\psi)\subset\mbox{Spl}_{F/F_{2},\mathbb{Q}}$
\item $\psi|_{\mathcal{O}_{E_{u}^{\times}}}=1$, where $u$ is the place
above $l$ induced by $\iota_{l}^{-1}\tau_{E}.$ 
\end{itemize}
\end{lem}
Define $\Pi^{1}:=\psi\otimes\Pi_{F}^{0}$, which is a cuspidal automorphic
representation of $GL_{1}(\mathbb{A}_{E})\times GL_{n}(\mathbb{A}_{F})$
and $\xi:=\iota_{l}\xi_{\mathbb{C}}$. 

Corresponding to the PEL datum $(F,*,V,\langle\cdot,\cdot\rangle,h)$
we have a PEL-type moduli problem of abelian varieties. This moduli
problem is defined in Section 2.1 of \cite{C} and here we recall
some facts about it. Since the reflex field of the PEL datum is $F'$,
the moduli problem for an open compact subgroup $U\subset G(\mathbb{A}^{\infty})$
is representable by a Shimura variety $X_{U}/F'$, which is a smooth
and quasi-projective scheme of dimension $2n-2$. The inverse system
of Shimura varieties $X_{U}$ as $U$ varies has an action of $G(\mathbb{A}^{\infty})$.
As in Section III.2 of \cite{HT}, starting with $\xi$, which is
an irreducible algebraic representation of $G$ over $\bar{\mathbb{Q}}_{l}$
we can define a lisse $\bar{\mathbb{Q}}_{l}$-sheaf $\mathcal{L}_{\xi}$
over each $X_{U}$ and the action of $G(\mathbb{A}^{\infty})$ extends
to the inverse system of sheaves. The direct limit \[
H^{i}(X,\mathcal{L}_{\xi}):=\lim_{\to}H^{i}(X_{U}\times_{F'}\bar{F}',\mathcal{L}_{\xi})\]
is a semisimple admissible representation of $G(\mathbb{A}^{\infty})$
with a continuous action of $\mbox{Gal}(\bar{F'}/F')$. It can be
decomposed as \[
H^{i}(X,\mathcal{L}_{\xi})=\bigoplus_{\pi}\pi\otimes R_{\xi,l}^{i}(\pi),\]
where the sum runs over irreducible admissible representations $\pi$
of $G(\mathbb{A}^{\infty})$ over $\bar{\mathbb{Q}}_{l}$. The $R_{\xi,l}^{i}(\pi)$
are finite dimensional continuous representations od $\mbox{Gal}(\bar{F}'/F')$
over $\bar{\mathbb{Q}}_{l}$. Let $\mathcal{A}_{U}$ be the universal
abelian variety over $X_{U}$, to the inverse system of which the
action of $G(\mathbb{A}^{\infty})$ extends. To the irreducible representation
$\xi$ of $G$ we can associate as in Section III.2 of \cite{HT}
non-negative integers $m_{\xi}$ and $t_{\xi}$ as well as an idempotent
$a_{\xi}$ of $H^{^{*}}(\mathcal{A}_{U}^{m_{\xi}}\times_{F'}\bar{F}',\bar{\mathbb{Q}}_{l}(t_{\xi}))$.
(Here $\mathcal{A}_{U}^{m_{\xi}}$ denotes the $m_{\xi}$-fold product
of $\mathcal{A}_{U}$ with itself over $X_{U}$ and $\bar{\mathbb{Q}}_{l}(t_{\xi})$
is a Tate twist.) We have an isomorphism\[
H^{i}(X_{U}\times_{F'}\bar{F}',\mathcal{L}_{\xi})\simeq a_{\xi}H^{i+m_{\xi}}(\mathcal{A}_{U}^{m_{\xi}}\times_{F'}\bar{F}',\mathbb{\bar{Q}}_{l}(t_{\xi})),\]
which commutes with the $G(\mathbb{A}^{\infty})$-action. 

For every finite place $v$ of $\mathbb{Q}$ we can define a base
change morphism taking certain admissible $G(\mathbb{Q}_{v})$-representations
to admissible $\mathbb{G}(\mathbb{Q}_{v})$-representations as in
Section 4.2 of \cite{Sh}. Recall that $\mbox{Ram}_{F/\mathbb{Q}}\cup\mbox{Ram}_{\mathbb{Q}}\Pi^{1}\subset\mbox{Spl}_{F/F_{2},\mathbb{Q}}$.
If $v\not\in\mbox{Spl}_{F/F_{2},\mathbb{Q}}$ then we can define the
morphism \[
BC:\mbox{Irr}_{(l)}^{\mathrm{ur}}(G(\mathbb{Q}_{v})\to\mbox{Irr}_{(l)}^{\mathrm{ur,},\theta-\mbox{st}}(\mathbb{G}(\mathbb{Q}_{v})),\]
taking unramified representations of $G(\mathbb{Q}_{v})$ to unramified,
$\theta$-stable representations of $\mathbb{G}(\mathbb{Q}_{v})$.
If $v\in\mbox{Spl}_{F/F_{2},\mathbb{Q}}$ then the morphism \[
BC:\mbox{Irr}_{(l)}(G(\mathbb{Q}_{v})\to\mbox{Irr}_{(l)}^{\theta-\mbox{st}}(\mathbb{G}(\mathbb{Q}_{v}))\]
can be defined explicitly since $G(\mathbb{Q}_{v})$ is split. Putting
these maps together we get for any finite set of primes $\mathfrak{S}_{\mathrm{fin}}$
such that \[
\mbox{Ram}_{F/\mathbb{Q}}\cup\mbox{Ram}_{\mathbb{Q}}(\Pi)\subset\mathfrak{S}_{\mathrm{fin}}\subset\mbox{Spl}_{F/F_{2},\mathbb{Q}}\]
a base change morphism \[
BC:\mbox{Irr}_{(l)}^{\mathrm{ur}}(G(\mathbb{A}^{\mathfrak{S}_{\mathrm{fin}}\cup\{\infty\}})\otimes\mbox{Irr}_{(l)}(G(\mathbb{A}_{\mathfrak{S}_{\mathrm{fin}}}))\to\mbox{Irr}_{(l)}^{\mathrm{ur,\theta-\mathrm{st}}}(\mathbb{G}(\mathbb{A}^{\mathfrak{S}_{\mathrm{fin}}\cup\{\infty\}})\otimes\mbox{Irr}_{(l)}^{\theta-\mathrm{st}}(\mathbb{G}(\mathbb{A}_{\mathfrak{S}_{\mathrm{fin}}})).\]

Let $p$ be a prime of $\mathbb{Q}$ which splits in $E$ and such
that there is a place of $F'$ above $p$ which splits in $F$. Let
$\mathfrak{S}_{\mathrm{fin}}$ be a finite set of primes such that
\[
\mbox{Ram}_{F/\mathbb{Q}}\cup\mbox{Ram}_{\mathbb{Q}}(\Pi)\cup\{p\}\subset\mathfrak{S}_{\mathrm{fin}}\subset\mbox{Spl}_{F/F_{2},\mathbb{Q}}\]
and set $\mathfrak{S}:=\mathfrak{S}_{\mathrm{fin}}\cup\{\infty\}$.
For any $R\in\mbox{Groth}(G(\mathbb{A}^{\mathfrak{S}})\times G(\mathbb{A}_{\mathfrak{S}_{\mathrm{fin}}})\times\mbox{Gal}(\bar{F}'/F'))$
(over $\bar{\mathbb{Q}}_{l}$) and $\pi^{\mathfrak{S}}\in\mbox{Irr}^{\mathrm{ur}}(G(\mathbb{A}^{\mathfrak{S}})$
define the $\pi^{\mathfrak{S}}$-isotypic part of $R$ to be \[
R\{\pi^{\mathfrak{S}}\}:=\sum_{\rho}n(\pi^{\mathfrak{S}}\otimes\rho)[\pi^{\mathfrak{S}}][\rho],\]
where $\rho$ runs over $\mbox{Irr}_{l}(G(\mathbb{A}_{\mathfrak{S}})\times\mbox{Gal}(\bar{F}'/F').$
Also define \[
R[\Pi^{1,\mathfrak{S}}]:=\sum_{\pi^{\mathfrak{S}}}R[\pi^{\mathfrak{S}}],\]
where each sum runs over $\pi^{\mathfrak{S}}\in\mbox{\mbox{Irr}}_{l}^{\mathrm{ur}}(G(\mathbb{A}^{\mathfrak{S}})$
such that $BC(\iota_{l}\pi^{\mathfrak{S}})\simeq\Pi^{1,\mathfrak{S}}$.
\begin{prop}
Let $\mathfrak{S}=\mathfrak{S}_{\mathrm{fin}}\cup\{\infty\}$ be as
above. We have the following equality \[
BC(H^{2n-2}(X,\mathcal{L}_{\xi})[\Pi^{1,\mathfrak{S}}])\simeq C_{G}[\iota_{l}^{-1}\Pi^{1,\infty}][R_{l}(\Pi_{F'}^{0})^{\otimes2}\otimes\mathrm{rec}_{l,\iota_{l}}(\psi)]\]
of elements of $\mbox{Groth}(G(\mathbb{A}^{\infty})\times\mbox{Gal}(\bar{F}'/F)$.
Here $C_{G}$ is a positive integer and $\mathrm{rec}_{l,\iota_{l}}(\psi)$
is the continuous $l$-adic character $\mathrm{Gal}(\bar{E}/E)\to\bar{\mathbb{Q}}_{l}^{\times}$
associated to $\psi$ by global class field theory. \end{prop}
\begin{proof}
Let $p\in\mathfrak{S}_{\mathrm{fin}}$ be a prime which splits in
$E$ and such that there is a place $w$ of $F'$ above the place
induced by $\tau_{E}$ over $p$ which splits in $F$, $w=w_{1}w_{2}$.
We start by recaling some constructions and results from Sections
2 and 5 of \cite{C}. It is possible to define an integral model of
each $X_{U}$ over the ring of integers $\mathcal{O}_{K}$ in $K:=F_{w_{1}}\simeq F_{w_{2}}$,
which itself represents a moduli problem of abelian varieties and
to which the sheaf $\mathcal{L}_{\xi}$. The special fiber $Y_{U}$
of this integral model has a stratification by open Newton polygon
strata $Y_{U,S,T}^{\circ}$, according to the formal (or etale) height
of the $p$-divisible group of the abelian variety at $w_{1}$ and
$w_{2}$. Each open Newton polygon stratum is covered by a tower of
Igusa varieties $\mathrm{Ig}_{U^{p},\vec{m}}^{(h_{1},h_{2})},$ where
$0\leq h_{1},h_{2}\leq n-1$ represent the etale heights of the $p$-divisible
groups at $w_{1}$ and $w_{2}$, and $\vec{m}$ is a tuple of positive
integers describing the level structure at $p$. 

Define \[
J^{(h_{1},h_{2})}(\mathbb{Q}_{p}):=\mathbb{Q}_{p}^{\times}\times D_{K,n-h_{1}}^{\times}\times GL_{h_{1}}(K)\times D_{K,n-h_{2}}^{\times}\times GL_{h_{2}}(K)\times\prod_{w}GL_{n}(F_{w}),\]
where $D_{K,n-h}$ is the division algebra over $K$ of invariant
$\frac{1}{n-h}$ and $w$ runs over places of $F$ above $\tau_{E}$
other than $w_{1}$ and $w_{2}$. The group $J^{(h_{1},h_{2})}(\mathbb{Q}_{p})$
acts on the directed system of $H_{c}^{j}(\mathrm{Ig}_{U^{p},\vec{m}}^{(h_{1},h_{2})},\mathcal{L}_{\xi})$,
as $U^{p}$ and $\vec{m}$ vary. Let \[
H_{c}(\mathrm{Ig}^{(h_{1},h_{2})},\mathcal{L}_{\xi})\in\mathrm{Groth}(G(\mathbb{A}^{\infty,p})\times J^{(h_{1},h_{2})})\]
be the alternating sum of the direct limit of $H_{c}^{j}(\mathrm{Ig}_{U^{p},\vec{m}}^{(h_{1},h_{2})},\mathcal{L}_{\xi})$
as in Section 5.1 of \cite{C}. Let $\pi_{p}\in\mathrm{Irr}_{l}(G(\mathbb{Q}_{p}))$
be a representation such that $BC(\pi_{p})\simeq\iota_{l}^{-1}\Pi_{p}^{1}$
(such a $\pi_{p}$ is unique up to isomorphism since $p$ splits in
$E$). Theorem 5.6 of \cite{C} gives a formula for computing the
cohomology of Igusa varieties, as elements of $\mbox{Groth}(\mathbb{G}(\mathbb{A}^{\mathfrak{S}})\times\mathbb{G}(\mathbb{A}_{\mathfrak{S}_{\mathrm{fin}}\backslash\{p\}})\times J^{(h_{1},h_{2})}(\mathbb{Q}_{p}))$:\[
BC^{p}(H_{c}(\mathrm{Ig}^{(h_{1},h_{2})},\mathcal{L}_{\xi})[\Pi^{1,\mathfrak{S}}])=\]
\begin{equation}
=e_{0}(-1)^{h_{1}+h_{2}}C_{G}[\iota_{l}^{-1}\Pi^{1,\mathfrak{S}}][\iota_{l}^{-1}\Pi_{\mathfrak{S}_{\mathrm{fin}}\backslash\{p\}}^{1}][\mathrm{Red}_{n}^{(h_{1},h_{2})}(\pi_{p})]\label{eq:igusa}\end{equation}
Here $e_{0}=\pm1$ independently of $h_{1},h_{2}$ and $\mathrm{Red}_{n}^{(h_{1},h_{2})}$
is a group morphism from $\mathrm{Groth}(G(\mathbb{Q}_{p}))$ to $\mathrm{Groth}(J^{(h_{1},h_{2})}(\mathbb{Q}_{p}))$,
defined explicitly above Theorem 5.6 of \cite{C}.

We can combine the above formula with Mantovan's formula for the cohomology
of Shimura varieties. This is the equality \begin{equation}
H(X,\mathcal{L}_{\xi})=\sum_{0\leq h_{1},h_{2}\leq n-1}(-1)^{h_{1}+h_{2}}\mathrm{Mant}_{(h_{1},h_{2})}(H_{c}(\mathrm{Ig}^{(h_{1},h_{2})},\mathcal{L}_{\xi}))\label{eq:mantovan}\end{equation}
of elements of $\mathrm{Groth}(G(\mathbb{A}^{\infty})\times W_{K})$.
Here $H(X,\mathcal{L}_{\xi})$ is the alternating sum of the direct
limit of the cohomology of the Shimura fibers (generic fibers) and
\[
\mathrm{Mant}_{(h_{1},h_{2})}:\mathrm{Groth}(J^{(h_{1},h_{2})}(\mathbb{Q}_{p}))\to\mathrm{Groth}(G(\mathbb{Q}_{p})\times W_{K})\]
is the functor defined in \cite{Man}. The formula \ref{eq:mantovan}
is what Theorem 22 of \cite{Man} amounts to in our situation, where
$h_{1}$ and $h_{2}$ are the parameters for the Newton stratification.
The extra term $(-1)^{h_{1}+h_{2}}$ occurs on the right hand side
because we use the same convention for the alternating sum of cohomology
as in \cite{C}, which differs by a sign from the conventions used
in \cite{Man} and \cite{Sh}. 

By combining formulas \ref{eq:igusa} and \ref{eq:mantovan} we get\[
BC^{p}(H(X,\mathcal{L}_{\xi})[\Pi^{1,\mathfrak{S}}])=e_{0}C_{G}[\iota_{l}^{-1}\Pi^{1,\infty,p}]\left(\sum_{0\leq h_{1},h_{2}\leq n-1}[\mathrm{Mant}_{(h_{1},h_{2})}(\mathrm{Red}_{n}^{(h_{1},h_{2})}(\pi_{p})]\right)\]
in $\mathrm{Groth}(\mathbb{G}(\mathbb{A}^{\infty,p})\times G(\mathbb{Q}_{p})\times W_{K})$.
We claim that \begin{equation}
\sum_{0\leq h_{1},h_{2}\leq n-1}[\mathrm{Mant}_{(h_{1},h_{2})}(\mathrm{Red}_{n}^{(h_{1},h_{2})}(\pi_{p})=[\pi_{p}][(\pi_{p,0}\circ\mathrm{Art}_{\mathbb{Q}_{p}}^{-1})|_{W_{K}}\otimes\iota_{l}^{-1}\mathcal{L}_{K,n}(\Pi_{F',w}^{0})].\label{eq:mantred}\end{equation}
By its definition above Theorem 5.6 of \cite{C}, the morphism $\mathrm{Red}_{n}^{(h_{1},h_{2})}(\pi_{p})$
breakes down as a product \[
(-1)^{h_{1}+h_{2}}\pi_{p,0}\otimes\mathrm{Red}^{n-h_{1},h_{1}}(\pi_{w_{1}})\otimes\mathrm{Red}^{n-h_{2},h_{2}}(\pi_{w_{2}})\otimes(\otimes_{w\not=w_{1},w_{2}}\pi_{w}),\]
where $w$ runs over places above the place of $p$ induced by $\tau_{E}$
other than $w_{1}$ and $w_{2}$. The morphism \[
\mathrm{Red}^{n-h,h}:\mathrm{Groth}(GL_{n}(K))\to\mathrm{Groth}(D_{K,n-h}^{\times}\times GL_{h}(K))\]
is also defined above Theorem 5.6 of \cite{C}. On the other hand,
the functor $\mathrm{Mant}_{(h_{1},h_{2})}$ also decomposes as a
product (see formula 5.6 of \cite{Sh}), into \[
\mathrm{Mant}_{(h_{1},h_{2})}(\rho)=\mathrm{Mant}_{1,0}(\rho_{0})\otimes\mathrm{Mant}_{n-h_{1},h_{1}}(\rho_{w_{1}})\otimes\mathrm{Mant}_{n-h_{2},h_{2}}(\rho_{w_{2}})\otimes(\otimes_{w\not=w_{1},w_{2}}\mathrm{Mant}_{0,m}(\rho_{w})),\]
where $w$ again runs over places above the place of $p$ induced
by $\tau_{E}$ other than $w_{1}$ and $w_{2}$. So\[
\sum_{0\leq h_{1},h_{2}\leq n-1}[\mathrm{Mant}_{(h_{1},h_{2})}(\mathrm{Red}_{n}^{(h_{1},h_{2})}(\pi_{p})]=\]
$=[\mathrm{Mant}_{1,0}(\pi_{p,0})]\otimes(\sum_{h_{1}=0}^{n-1}(-1)^{h_{1}}[\mathrm{Mant}_{n-h_{1},h_{1}}(\mathrm{Red}^{n-h_{1},h_{1}}(\pi_{w_{1}}))])\otimes$\[
\otimes(\sum_{h_{2}=0}^{n-1}(-1)^{h_{2}}[\mathrm{Mant}_{n-h_{2},h_{2}}(\mathrm{Red}^{n-h_{2},h_{2}}(\pi_{w_{2}}))])\otimes(\otimes_{w\not=w_{1},w_{2}}[\pi_{w}]).\]
Now by applying Prop. 2.2.(i) and 2.3 of \cite{Sh} we get the desired
result (note that the normalization used in their statements is slightly
different than ours, but the relation between the two different normalizations
is explained above the statement of Prop. 2.3). 

Applying equation \ref{eq:mantred}, we first see that \begin{equation}
BC(H(X,\mathcal{L}_{\xi})[\Pi^{1,\mathfrak{S}}])=e_{0}C_{G}[\iota_{l}^{-1}\Pi^{1,\infty}][(\pi_{p,0}\circ\mathrm{Art}_{\mathbb{Q}_{p}}^{-1})|_{W_{K}}\otimes\iota_{l}^{-1}\mathcal{L}_{K,n}(\Pi_{F',w}^{0})]\label{eq:local}\end{equation}
in $\mathrm{Groth}(\mathbb{G}(\mathbb{A}^{\infty})\times W_{K})$,
which means that \[
BC(H(X,\mathcal{L}_{\xi})[\Pi^{1,\mathfrak{S}}])=e_{0}[\iota_{l}^{-1}\Pi^{1,\infty}][R'(\Pi^{1})],\]
for some $[R'(\Pi^{1})]\in\mathrm{Groth}(\mathrm{Gal}(\bar{F}'/F))$.
We show now that \[
[R'(\Pi^{1})]=C_{G}[R(\Pi_{F'}^{0})^{\otimes2}\otimes\mathrm{rec}_{l,\iota_{l}}(\psi)]\]
in $\mathrm{Groth}(\mathrm{Gal}(\bar{F}/F'))$ using the Chebotarev
density theorem. Note first that $R'(\Pi^{1})$ is simply the sum
of (the alternating sum of) $R_{\xi,l}^{k}(\pi^{\infty})$ where $\pi^{\infty}$
runs over $\mathrm{Irr}_{l}(G(\mathbb{A}^{\infty})$ such that 
\begin{itemize}
\item $BC(\iota_{l}\pi^{\mathfrak{S}})\simeq\Pi^{1,\mathfrak{S}}$
\item $BC(\iota_{l}\pi_{\mathfrak{S}_{\mathrm{fin}}})\simeq\Pi_{\mathfrak{S}_{\mathrm{fin}}}$
\item $R_{\xi,l}^{k}(\pi^{\infty})\not=0$ for some $k$.
\end{itemize}
The set of such $\pi$ doesn't depend on $\mathfrak{S}$ if $\mathfrak{S}$
is chosen as described above this proposition, so the Galois representation
$R^{'}(\Pi^{1})$ is also independent of $\mathfrak{S}$. Therefore,
for any prime $w_{1}$ of $F$ where $\Pi^{1}$ is unramfied and which
is above a prime $w$ of $F'$ which splits in $F$ and above a prime
$p\not=l$ of $\mathbb{Q}$ which splits in $E$, we can choose a
finite set of places $\mathfrak{S}$ containing $p$ such that we
get from equation \ref{eq:local} \[
[R'(\Pi^{1})|_{W_{F_{w_{1}}}}]=C_{G}[(R(\Pi_{F'}^{0})^{\otimes2}\otimes\mathrm{rec}_{l,\iota_{l}}(\psi))_{W_{F_{w_{1}}}}].\]
By the Cebotarev density theorem (which tells us the Frobenius elements
of primes $w_{1}$ are dense in $\mathrm{Gal}(\bar{F}'/F)$) we conclude
that \[
[R'(\Pi^{1})]=C_{G}[R(\Pi_{F'}^{0})^{\otimes2}\otimes\mathrm{rec}_{l,\iota_{l}}(\psi)]\]
in $\mathrm{Groth}(\mathrm{Gal}(\bar{F}/F'))$.

It remains to see that $e_{0}=1$ and that $H^{k}(X,\mathcal{L}_{\xi})[\Pi^{1,\mathfrak{S}}]=0$
unless $k=2n-2$. In fact, it suffices to show the latter, since then
$H(X,\mathcal{L}_{\xi})[\Pi^{1,\mathfrak{S}}]$ will have to be an
actual representation, so that would force $e_{0}=1$. The fact that
$H^{k}(X,\mathcal{L}_{\xi})[\Pi^{1,\mathfrak{S}}]=0$ for $k\not=2n-2$
can be seen as in the proof of Corollary 7.3 of \cite{C} by choosing
a prime $p\not=l$ to work with and applying the spectral sequences
in Prop. 7.2 of loc. cit. and noting that the terms of those spectral
sequence are $0$ outside the diagonal corresponding to $k=2n-2$. \end{proof}
\begin{cor}
\label{reduction to cohomology}By Lemmas 1.4 and 1.7 of \cite{TY}
and by the same argument as in the proof of Theorem 7.4 of \cite{C},
in order to show that \[
WD(R_{l}(\Pi_{F'}^{0})|_{Gal(\bar{F'_{\mathfrak{p}}}_{\mathfrak{}}/F'_{\mathfrak{p}})})^{F-ss}\]
is pure, it suffices to show that \[
WD(BC(H^{2n-2}(X,\mathcal{L}_{\xi})[\Pi^{\mathfrak{S}}])|_{\mathrm{Gal}(\bar{F'_{\mathfrak{p}}}/F'_{\mathfrak{p}})})^{F-ss}\]
is pure, where $\mathfrak{S}$ is chosen such that it contains $l$. 
\end{cor}
Now recall that $\mathfrak{p}$ is a place of $F'$ above $l$ and
such that $\mathfrak{p}=\mathfrak{p}_{1}\mathfrak{p}_{2}$. From now
on, set $K:=F_{\mathfrak{p}_{1}}\simeq F_{\mathfrak{p}_{2}}$, where
the isomorphism is via $\sigma$. Let $\mathcal{O}_{K}$ be the ring
of integers in $K$ with uniformizer $\varpi$ and residue field $k$.
For $i=1,2$ let $\mathrm{Iw}_{n,\mathfrak{p}_{i}}$ be the subgroup
of matrices in $GL_{n}(\mathcal{O}_{K})$ which reduce modulo $\mathfrak{p}_{i}$
to the Borel subgroup $B_{n}(k)$. Now we set \[
U_{\mathrm{Iw}}=U^{l}\times U_{l}^{\mathfrak{p}_{1},\mathfrak{p}_{2}}(m)\times\mathrm{Iw}_{n,\mathfrak{p}_{1}}\times\mathrm{Iw}_{n,\mathfrak{p}_{2}}\subset G(\mathbb{A}^{\infty}),\]
for some $U^{l}\subset G(\mathbb{A^{\infty}})$ compact open and $U_{l}^{\mathfrak{p}_{1},\mathfrak{p}_{2}}$
congruence subgroup at $l$ away from $\mathfrak{p}_{1}$ and $\mathfrak{p}_{2}$.
In Section 2.2 of \cite{C}, an integral model for $X_{U_{\mathrm{Iw}}}/\mathcal{O}_{K}$
is defined. This is a proper scheme of dimension $2n-1$ with smooth
generic fiber. The special fiber $Y_{U_{\mathrm{Iw}}}$ has a stratification
by closed Newton polygon strata $Y_{U_{\mathrm{Iw}},S,T}$ with $S,T\subseteq\{1,\dots,n\}$
non-empty subsets. These strata are proper, smooth schemes over $k$
of dimension $2n-\#S-\#T$. In fact, \[
Y_{U_{\mathrm{Iw}},S,T}=(\bigcap_{i\in S}Y_{1,i})\cap(\bigcap_{j\in T}Y_{2,j}),\]
where each $Y_{i,j}$ for $i=1,2$ and $j=1,\dots,n$ is cut out by
one local equation. We can also define \[
Y_{U_{\mathrm{Iw}}}^{(l_{1},l_{2})}=\bigsqcup_{\substack{S,T\subseteq\{1,\dots,n\}\\
\#S=l_{1}\\
\#T=l_{2}}
}Y_{U_{\mathrm{Iw}},S,T}\]
By Prop. 2.8 of \cite{C}, the completed local rings of $X_{U_{\mathrm{Iw}}}$
at closed geometric points $s$ of $X_{U_{\mathrm{Iw}}}$ are isomorphic
to \[
\mathcal{O}\hat{\ }_{X_{U_{\mathrm{Iw}}},s}\simeq W_{(K)}[[X_{1},\dots,X_{n},Y_{1},\dots Y_{n}]]/(X_{i_{1}}\cdot\dots\cdot X_{i_{r}}-\varpi,Y_{j_{1}}\cdot\dots\cdot Y_{j_{s}}-\varpi),\]
where $\{i_{1},\dots,i_{r}\}\subseteq\{1,\dots,n\},$ $\{j_{1},\dots,j_{r}\}\subseteq\{1,\dots,n\}$
and $W_{(K)}$ is the ring of integers in the completion of the maximal
unramified extension of $K$. The closed subscheme $Y_{1,i_{l}}$
is cut out in $\mathcal{O}\hat{\ }_{X_{U_{\mathrm{Iw}}},s}$ by $X_{i_{l}}=0$
and $Y_{2,j_{l}}$ is cut out by $Y_{j_{l}}=0$. 

The action of $G(\mathbb{A}^{\infty,p})$ extends to the inverse system
$X_{U_{\mathrm{Iw}}}/\mathcal{O}_{K}$. There is a universal abelian
variety $\mathcal{A}_{U_{\mathrm{Iw}}}/\mathcal{O}_{K}$ and the actions
of $G(\mathbb{A}^{\infty})$ and $a_{\xi}$ extend to it. We can define
a stratification of the special fiber of $\mathcal{A}_{U_{\mathrm{Iw}}}$
by \[
\mathcal{A}_{U_{\mathrm{Iw}},S,T}=\mathcal{A}_{U_{\mathrm{Iw}}}\times_{X_{U_{\mathrm{Iw}}}}X_{U_{\mathrm{Iw}},S,T}.\]
Moreover, $\mathcal{A}_{U_{\mathrm{Iw}}}^{m_{\xi}}$ and $\mathcal{A}_{U_{\mathrm{Iw}}}$
and with respect to the special fiber stratification satisfies the
same geometric properties as $X_{U_{\mathrm{Iw}}}$. In particular,
we shall see in the next section (or it follows from Section 3 of
\cite{C}) that it follows from these properties that $\mathcal{A}_{U_{\mathrm{Iw}}}^{m_{\xi}}$
can be endowed with a vertical logarithmic structure $M$ such that
\[
(\mathcal{A}_{U_{\mathrm{Iw}}}^{m_{\xi}},M)\to(\mbox{Spec }\mathcal{O}_{K},\mathbb{N})\]
is log smooth, where $(\mbox{Spec }\mathcal{O}_{K},\mathbb{N})$ is
the canonical log structure associated to the closed point. Also,
we'll see that its special fiber is of Cartier type. This means that
we can define the log crystalline cohomology of $(\mathcal{A}_{U_{\mathrm{Iw}}}^{m_{\xi}},M)$.
Indeed, if $W=W(k)$ is the ring of Witt vectors of $k$, then we
let \[
H_{\mathrm{cris}}^{*}(\mathcal{A}_{U_{\mathrm{Iw}}}^{m_{\xi}}/W)\]
be the log crystalline cohomology of $(\mathcal{A}_{U_{\mathrm{Iw}}}^{m_{\xi}}\times_{\mathcal{O}_{K}}k,M)$
(here we suppressed $M$ from the notation). This also has an action
of $a_{\xi}$ as an idempotent and of $G(\mathbb{A}^{\mathfrak{S}})$.
From the isomorphis \[
H^{2n-2}(X,\mathcal{L}_{\xi})\simeq a_{\xi}H^{2n-2+m_{\xi}}(\mathcal{A}^{m_{\xi}},\mathbb{\bar{Q}}_{l}(t_{\xi}))\]
and Corollary \ref{reduction to cohomology}, we see that it is enough
to show that \[
a_{\xi}WD(H^{2n-2+m_{\xi}}(\mathcal{A}^{m_{\xi}},\bar{\mathbb{Q}}_{l}(t_{\xi})|_{\mathrm{Gal}(\bar{K}/K)})[\Pi^{1,\mathfrak{S}}]\]
is pure. Let $\tau_{0}:W\hookrightarrow\bar{\mathbb{Q}}_{l}$ be an
embedding over $\mathbb{Z}_{l}$. By the semistable comparison theorem
of \cite{N}, we have \[
\lim_{\substack{\to\\
U_{\mathrm{Iw}}}
}a_{\xi}(H_{\mathrm{cris}}^{2n-2+m_{\xi}}(\mathcal{A}_{U_{\mathrm{Iw}}}^{m_{\xi}}\times_{\mathcal{O}_{K}}k/W)\otimes_{W,\tau_{0}}\bar{\mathbb{Q}}_{l}(t_{\xi}))[\Pi^{1,\mathfrak{S}}]\simeq\lim_{\substack{\to\\
U_{\mathrm{Iw}}}
}a_{\xi}WD(H^{2n-2+m_{\xi}}(\mathcal{A}^{m_{\xi}}\times_{\mathcal{O}_{K}}\bar{K},\bar{\mathbb{Q}}_{l}(t_{\xi})|_{\mathrm{Gal}(\bar{K}/K)})[\Pi^{1,\mathfrak{S}}],\]
so it suffices to understand the (direct limit of the) log crystalline
cohomology of the special fiber of $\mathcal{A}_{U_{\mathrm{Iw}}}^{m_{\xi}}$.
Note that in order to apply this theorem we need to check that $(\mathcal{A}_{U_{\mathrm{Iw}}}^{m_{\xi}},M)$
is a fine and saturated log-smooth proper vertical $(\mbox{Spec }\mathcal{O}_{K},\mathbb{N})$-scheme
and such that its special fiber is of Cartier type. All these properties
follow immediately from the explicit description of the log structure
$M$ in Section 3.

\section{Log crystalline cohomology }

\subsection{Log structures}

Let $\mathcal{O}_{K}$ be the ring of integers in a finite extension
$K$ of $\mathbb{Q}_{p}$ ($p$ is some prime number, which is meant
to be identified with $l$), with uniformizer $\varpi$ and residue
field $k$. Let $W=W(k)$ be the ring of Witt vectors of $k$, with
$W_{n}=W_{n}(k)$ referring to the Witt vectors of length $n$ over
$k$. Let $W_{(K)}$ be the ring of integers in the completion of
the maximal unramified extension of $K$. 

Let $X/\mathcal{O}_{K}$ be a locally Noetherian scheme such that
the completions of the strict henselizations $\mathcal{O}_{X,s}^{\wedge}$
at closed geometric points $s$ of $X$ are isomorphic to \[
W_{(K)}[[X_{1},\dots,X_{n},Y_{1},\dots Y_{n},Z_{1},\dots,Z_{m}]]/(X_{i_{1}}\cdot\dots\cdot X_{i_{r}}-\varpi,Y_{j_{1}}\cdot\dots\cdot Y_{j_{s}}-\varpi)\]
for some indices $i_{1},\dots,i_{r},j_{1},\dots,j_{s}\in\{1,\dots n\}$
and some $1\leq r,s\leq n$. Also assume that the special fiber $Y$
is a union of closed subschemes $Y{}_{1,j}$ with $j\in\{1,\dots n\}$,
which are cut out by one local equation, such that if $s$ is a closed
geometric point of $Y{}_{1,j}$, then $j\in\{i_{1},\dots,i_{r}\}$
and $Y{}_{1,j}$ is cut out in $\mathcal{O}_{X,s}^{\wedge}$ by the
equation $X_{j}=0$. Similarly, assume that $Y$ is a union of closed
subschemes $Y{}_{2,j}$ with $j\in\{1,\dots,n\}$, which are cut out
by one local equation such that if $s$ is a closed geometric point
of $Y_{2,j}$ then $j\in\{j_{1},\dots,j_{r}\}$ and $Y{}_{2,j}$ is
cut out in $\mathcal{O}_{X',s}^{\wedge}$ by the equation $Y_{j}=0$.
Then, by Lemma 2.9 of \cite{C}, $X$ is locally etale over \[
X_{r,s,m}=\mbox{Spec }\mathcal{O}_{K}[X_{1},\dots,X_{n},Y_{1},\dots,Y_{n},Z_{1},\dots Z_{m}]/(X_{1}\cdot\dots\cdot X_{r}-\varpi,Y_{1}\cdot\dots\cdot Y_{s}-\varpi).\]
The closed subschemes $Y_{i,j}$ for $i=1,2$ and $j=1,\dots,n$ are
Cartier divisors, which in the local model $X_{r,s,m}$ correspond
to the divisors $X_{j}=0$ or $Y_{j}=0$. 

Let $Y/k$ be the special fiber of $X$. For $1\leq i,j\leq n$ we
define $Y^{(i,j)}$ to be the disjoint union of the closed subschemes
of $Y$ \[
(Y_{1,l_{1}}\cap\dots\cap Y_{1,l_{i}})\bigcap(Y_{2,m_{1}}\cap\dots\cap Y_{2,m_{j}}),\]
as $\{l_{1},\dots,l_{i}\}$ (resp. $\{m_{1},\dots,m_{j}\}$) range
over subsets of $\{1,\dots,n\}$ of cardinality $i$ (resp. $j$).
Each $Y^{(i,j)}$ is a proper smooth scheme over $k$ of dimension
$2n-i-j$. 
\begin{rem}
Even though this section is general, we are basically thinking of
$X$ as $\mathcal{A}_{U_{\mathrm{Iw}}}$ for some compact open subgroup
$U_{\mathrm{Iw}}\subset G(\mathbb{A}^{\infty})$ with Iwahori level
structure at $\mathfrak{p}_{1}$ and $\mathfrak{p}_{2}$. $X_{U_{\mathrm{Iw}}}$
(and therefore $\mathcal{A}_{U_{\mathrm{Iw}}}$ as well) satisfies
the above conditions by Prop. 2.8 of \cite{C}. The prime $p$ is
meant to be identified with $l$. 
\end{rem}
Let $(\mbox{Spec }\mathcal{O}_{K},\mathbb{N})$ be the log scheme
corresponding to $\mbox{Spec }\mathcal{O}_{K}$ endowed with the canonical
log structure associated to the special fiber. This is given by the
map $1\in\mathbb{N}\mapsto\varpi\in\mathcal{O}_{K}$. We endow $X$
with the log structure $M$ associated to the special fiber $Y$.
Let $j:X_{K}\to X$ be the open immersion and $i:Y\to X$ be the closed
immersion. This log structure is defined by \[
M=j_{*}(\mathcal{O}_{X_{K}}^{\times})\cap\mathcal{O}_{X}\to\mathcal{O}_{X}.\]
We have a map of log schemes $(X,M)\to(\mbox{Spec }\mathcal{O}_{K},\mathbb{N})$,
given by sending $1\in\mathbb{N}$ to $\varpi\in M$. Locally, we
have a chart for this map, given by \[
\mathbb{N}\to\mathbb{N}^{r}\oplus\mathbb{N}^{s}/(1,\dots,1,0,\dots,0)=(0,\dots,0,1,\dots1),\]
\[
1\mapsto(1,\dots,1,0,\dots,0)=(0,\dots,0,1\dots1).\]
It is easy to see from this that $(X,M)/(\mbox{Spec }\mathcal{O}_{K},\mathbb{N})$
is log smooth and that the log structure $M$ on $X$ is fine, saturated
and vertical. We can pull back $M$ to a log structure on $Y$, which
we still denote $M$ and then we get a log smooth map of log schemes
\[
(Y,M)\to(\mbox{Spec }k,\mathbb{N}).\]
(Here we have the canonical log structure on $k$ associated to $1\in\mathbb{N}\mapsto0\in k,$
which is the same as the pullback of the canonical log structure on
$\mbox{Spec }\mathcal{O}_{K}$.) Note that, since $(X,M)$ is saturated
over $(\mathrm{Spec}\ \mathcal{O}_{K},\mathbb{N})$, so its special
fiber is of Cartier type (cf. \cite{Ts}). 

We can also endow $X$ with log structures $\tilde{M}_{1}$,$\tilde{M}_{2}$
and $\tilde{M}$. Let $U_{i,j}$ be the complement of $Y_{i,,j}$
in $X$ for $i=1,2$ and $j=1,\dots,n$. Let \[
j_{i,j}:U_{i,j}\to X\]
denote the open immersion. We define $\tilde{M}_{1}$,$\tilde{M}_{2}$
and $\tilde{M}$ as follows\[
\tilde{M}_{1}=\left(\bigoplus_{j=1}^{n}(j_{1,j*}(\mathcal{O}_{U_{1,j}}^{\times})\cap\mathcal{O}_{X})\right)/\sim\]
\[
\tilde{M}_{2}=\left(\bigoplus_{j=1}^{n}(j_{1,j*}(\mathcal{O}_{U_{1,j}}^{\times})\cap\mathcal{O}_{X})\right)/\sim\]
\[
\tilde{M}=\left(\bigoplus_{j=1}^{n}(j_{1,j*}(\mathcal{O}_{U_{1,j}}^{\times})\cap\mathcal{O}_{X})\oplus\bigoplus_{j=1}^{n}(j_{2,j*}(\mathcal{O}_{U_{2,j}}^{\times})\cap\mathcal{O}_{X})\right)/\sim,\]
where $\sim$ signifies that we've identified the image of $\mathcal{O}_{X}^{\times}$
in all the terms of the direct sums (basically we are taking an amalgamated
sum of the log structures associated to each of the $Y_{i,j}$). We
have a map $\tilde{M}\to M$ given by inclusion on each $\mathcal{O}_{U_{i,j}}^{\times}$. 
\begin{lem}
\label{chart for M tilde}Locally on $X$, we have a chart for $\tilde{M}$
given by \[
X\to\mbox{Spec }\mathcal{O}_{k}[X_{1},\dots,X_{n},Y_{1},\dots,Y_{n},Z_{1},\dots Z_{m}]/(X_{1}\cdot\dots\cdot X_{r}-\varpi,Y_{1}\cdot\dots\cdot Y_{r}-\varpi)\to\mbox{Spec }\mathbb{Z}[\mathbb{N}^{r}\oplus\mathbb{N}^{s}],\]
where $(0,\dots0,1,0\dots0)\mapsto X_{i}$ if the $1$ is in the $i$th
position and $1\leq i\leq r$ and $(0,\dots0,1,0\dots0)\mapsto Y_{i-r}$
if the $1$ is in the $i$th position and $r+1\leq i\leq r+s.$\end{lem}
\begin{proof}
We shall make use of Kato-Niziol's results on log smoothness and log
regularity, namely:
\begin{itemize}
\item if $f:T\to S$ is a log smooth morphism of fs log schemes with $S$
log regular then $T$ is log regular (see 8.2 of\cite{K2}) and
\item if $T$ is log regular, then $M_{T}=j_{*}\mathcal{O}_{U}^{\times}\cap\mathcal{O}_{X}$,
where $j:U\hookrightarrow T$ is the inclusion of the open subset
of triviality of $T$ (see 8.6 of \cite{N2}). 
\end{itemize}
Let us define the following log schemes over $(\mathrm{Spec}\ \mathcal{O}_{K},triv)$:
\[
\tilde{U}:=\mathrm{Spec}\ \mathcal{O}_{K}[X_{1},\dots,X_{n},\sigma]/(X_{1}\cdot\dots\cdot X_{r}-\sigma)\]
\[
\tilde{V}:=\mathrm{Spec}\ \mathcal{O}_{K}[Y_{1},\dots,Y_{n},\tau]/(Y_{1}\cdot\dots\cdot Y_{s}-\tau)\]
\[
W:=\mathrm{Spec}\ \mathcal{O}_{K}[Z_{1},\dots,Z_{m}]\]

\[
Z:=\tilde{U}\times_{\mathrm{(Spec}\ \mathcal{O}_{K},triv)}\tilde{V}\times_{\mathrm{(Spec}\ \mathcal{O}_{K},triv)}W\]
Then $Z$, equipped with the product log structure $L$ is smooth
over $\mathcal{O}_{K}$ and log smooth over $(\mathrm{Spec}\ \mathcal{O}_{K}[\sigma,\tau],triv)$.
Therefore, $Z$ is regular. The log structure $L$ is given by the
simple normal crossings divisor \[
D:=(\bigcup_{j=1}^{r}(X_{j}=0))\cup(\bigcup_{j=1}^{s}(Y_{j}=0)).\]
Since $Z$ is regular, the log structure $L$ is the same as the amalgamation
of the log structures defined by the smooth divisors $(X_{j}=0),(Y_{j}=0)$.
Locally on $X$, we have a commutative diagram of schemes with a cartesian
square \begin{equation}
\xymatrix{X\ar[r] & X_{r,s,m}\ar[d]\ar[r] & Z\ar[d]\\
 & \mathrm{Spec}\ \mathcal{O}_{K}\ar[r] & \mathrm{Spec}\ \mathcal{O}_{K}[\tau,\sigma]}
,\label{diagram for M tilde}\end{equation}
where the inverse image of $(X_{j}=0)$ in $X$ is $Y_{j}^{1}$, the
inverse image of $(Y_{j}=0)$ in $X$ is $Y_{j}^{2}$. Therefore,
the log structure on $X$ induced by that of $Z$ coincides with the
log structure $\tilde{M}$, defined as the amalgamated sum of the
log structures induced by the $Y_{j}^{1}$ and $Y_{j}^{2}$. 
\end{proof}
If we endow $\mbox{Spec }\mathcal{O}_{K}$ with the log structure
$\mathbb{N}^{2}$ associated to $(a,b)\in\mathbb{N}^{2}\mapsto\pi^{a+b}\in\mathcal{O}_{K}$,
then we claim that we have a log smooth map of log schemes \begin{equation}
(X,\tilde{M})\to(\mbox{Spec }\mathcal{O}_{K},\mathbb{N}^{2})\label{critical map}\end{equation}
whose chart is given locally by \[
(a,b)\in\mathbb{N}^{2}\mapsto(a,\dots a,b,\dots b)\in\mathbb{N}^{r}\oplus\mathbb{N}^{s}.\]

By definition, $\tilde{M}$ is the amalgamated sum of $\tilde{M}_{1}$
and $\tilde{M}_{2}$ as log structures on $X$ (or, in other words,
$\tilde{M}$ is the log structure associated to the pre-log structure
$\tilde{M}_{1}\oplus\tilde{M}_{2}\to\mathcal{O}_{X}$). Therefore,
it suffices to prove the following lemma. 
\begin{lem}
We can define a global map of log schemes $(X,\tilde{M}_{1})\to(\mathrm{Spec}\ \mathcal{O}_{K},\mathbb{N})$
which locally admits the chart given by the diagonal embedding $\mathbb{N}\to\mathbb{N}^{r}$.\end{lem}
\begin{proof}
It suffices to show that $\varpi$ is a global section of $\tilde{M}_{1}$,
since then we can simply map $1\in\mathbb{N}$ to $\varpi\in\tilde{M}_{1}$.
For this, note that we have a natural map of log structures on $X$
\[
\tilde{M}_{1}\to M,\]
since the open subset of triviality of $\tilde{M}_{1}$ is the generic
fiber of $X$ and $M$ is the log structure defined by the inclusion
of the generic fiber. Moreover, we can check locally that this map
is injective, since it can be described by the chart $\mathbb{N}^{r}\to\mathbb{N}^{r}\oplus\mathbb{N}^{s}\to(\mathbb{N}^{r}\oplus\mathbb{N}^{s})/\mathbb{N}$
for $r,s\geq1$, where the first map is the identity on the first
factor. Now, locally on $X$ we have the equation $X_{1}\cdot\dots\cdot X_{r}=\varpi$,
where $X_{i}$ are local equations defining the closed subschemes
$Y_{i}^{1}$ of $X$. By definition, the $X_{i}$ are local sections
of $\tilde{M}_{1}$, so $\varpi$ is a local section of $\tilde{M}_{1}$.
But $\varpi$ is also a global section of $M$ and $\tilde{M}_{1}\hookrightarrow M$,
so $\varpi$ is a global section of $\tilde{M}_{1}$. \end{proof}
\begin{lem}
We have a cartesian diagram of maps of log schemes \[
\xymatrix{(X,M)\ar[r]\ar[d] & (X,\tilde{M})\ar[d]\\
(\mbox{Spec }\mathcal{O}_{K},\mathbb{N})\ar[r] & (\mbox{Spec }\mathcal{O}_{K},\mathbb{N}^{2})}
,\]
where the bottom horizontal arrow is the identity on the underlying
schemes and maps $(a,b)\in\mathbb{N}^{2}$ to $a+b\in\mathbb{N}.$ \end{lem}
\begin{proof}
We go back to the notation used in the proof of Lemma \ref{chart for M tilde}.
Locally on $X$, we have the following commutative diagram of log
schemes \[
\xymatrix{(X,M)\ar[d]\ar[r] & \tilde{U}\times_{\mathrm{Spec}\ \mathcal{O}_{K}[u]}\tilde{V}\times W\ar[r]\ar[d] & Z\ar[d]\\
\mathrm{(Spec}\ \mathcal{O}_{K},\mathbb{N})\ar[r] & \mathrm{(Spec}\ \mathcal{O}_{K}[u],\mathbb{N})\ar[r] & \mathrm{(Spec}\ \mathcal{O}_{K}[\tau,\sigma],\mathbb{N}^{2})}
,\]
where in the bottom row both $\tau$ and $\sigma$ are mapped to $u$,
which is in turn mapped to $0$. The second square is cartesian and
the horizontal maps in it are closed, but not exact, immersions. The
first bottom map is an exact closed immersion, while the first top
map is the composition of an etale morphism with an exact closed immersion.
The lemma follows from the commutative diagram (\ref{diagram for M tilde})
and the above diagram.
\end{proof}

\subsection{Variations on the logarithmic de Rham-Witt complex}

Define the pre-log structure $\mathbb{N}^{2}\to W_{n}[\tau,\sigma]$
given by $(a,b)\mapsto\tau^{a}\sigma^{b}$. By abuse of notation,
we write $\mbox{(Spec }W_{n}[\tau,\sigma],\mathbb{N}^{2})$ for the
log scheme endowed with the associated log structure. We have the
composite map of log schemes \[
(Y,\tilde{M})\to(\mbox{Spec }k,\mathbb{N}^{2})\to(\mbox{Spec }W_{n}[\tau,\sigma],\mathbb{N}^{2}),\]
where $\mathbb{N}^{2}\to\mathbb{N}^{2}$ is the obvious isomorphism.
We shall call $(Z,\tilde{N})$ a \emph{lifting} for this morphism
if $(Z,\tilde{N})$ is a fine log scheme such that the composite map
$(Y,\tilde{M})\to(\mbox{Spec }W_{n}[\tau,\sigma],\mathbb{N}^{2})$
factors through $f:(Y,\tilde{M})\to(Z,\tilde{N})$, which is a closed
immersion and a map $(Z,\tilde{N})\to(\mbox{Spec }W_{n}[\tau,\sigma],\mathbb{N}^{2})$,
which is log smooth. Such liftings always exists locally on $Y$ and
give rise to embedding systems as defined in paragraph 2.18 of \cite{HK}.
If $(U,\tilde{M}_{U})\to(Y,\tilde{M})$ is a covering and $(Z,\tilde{N})$
is a lifting for $(U,\tilde{M}_{U})\to(\mbox{Spec }W_{n}[\tau,\sigma]),\mathbb{N}^{2})$,
then we may define an embedding system $((U^{i},\tilde{M}_{U}^{i}),(Z^{i},\tilde{N}^{i}))$
for $(Y,\tilde{M})\to(\mbox{Spec }W_{n}[\tau,\sigma],\mathbb{N}^{2})$
by taking the fiber product of $i+1$ copies of $U$ over $Y$ and
of $i+1$ copies of $(Z,\tilde{N})$ over $(\mbox{Spec }W_{n}[\tau,\sigma],\mathbb{N}^{2})$.
Since $(Y,\tilde{M})$ is an fs log scheme, we may assume the same
for the local lifting $(Z,\tilde{N})$. 

Let $C_{(Y,\tilde{M})/(W_{n},triv)}$ be the crystalline complex associated
to the embedding system obtained from local loftings $(Z^{\cdot},\tilde{N}^{\cdot})$
and define \[
\tilde{\tilde{C}}_{Y}:=C_{(Y,\tilde{M})/(W_{n},triv)}\otimes_{W_{n}<\tau,\sigma>}W_{n}.\]

Let $\mbox{Spec }W_{n}[u]$ be endowed with the log structure associated
to $1\in\mathbb{N}\mapsto u\in W_{n}[u]$. Consider the map of log
schemes $G:(\mbox{Spec }W_{n}[u],\mathbb{N})\to(\mbox{Spec }W_{n}[\tau.\sigma],\mathbb{N}^{2})$
given by $\tau,\sigma\mapsto u$ and $(a,b)\in\mathbb{N}^{2}\mapsto a+b\in\mathbb{N}$.
The pullback of $(Y,\tilde{M})$ along $G$ is just $(Y,M)$. Let
$(Z',N')$ be the pullback of $(Z,\tilde{N})$ along $G$, equipped
with a map $f':(Y',M')\to(Z',N')$, which is the pullback of $f$.
Then $(Z',N')$ is a (local) lifting for $(Y,M)\to(\mbox{Spec }W_{n}[u],\mathbb{N})$,
and gives rise to an embedding system for this morphism. Indeed, what
we need to check is that $(Z',N')\to(\mbox{Spec }W_{n}[u],\mathbb{N})$
is log smooth and that $f'$ is a closed immersion of log schemes.
For the first we note that log smoothness is preserved under base
change in the category of log schemes and that \[
(Z',N')=(((Z,\tilde{N})\times_{G}(\mbox{Spec }W_{n}[u],\mathbb{N}))^{\mathrm{int}})^{\mathrm{sat}}\to(Z,\tilde{N})\times_{G}(\mbox{Spec }W_{n}[u],\mathbb{N})\]
is log smooth. We also note that $g:Y\to(Z\times_{\mbox{Spec }W_{n}[\tau,\sigma]}\mbox{Spec }W_{n}[u])$
is a closed immersion, since $Y\to Z$ is a closed immersion. The
morphism of schemes $Z'\to(Z\times_{\mbox{Spec }W_{n}[\tau,\sigma]}\mbox{Spec }W_{n}[u])$
is a composition of a finite morphism with a closed immersion, so
$Y\to Z'$ is a closed immersion as well. Also, $g^{*}(\tilde{N}\oplus_{\mathbb{N}^{2}}\mathbb{N})\to M$
is surjective and factors through $(f')^{*}(N')\to M$, so $(f')^{*}(N')\to M$
is surjective as well. 

We now follow the constructions in section 3.6 of \cite{HK} using
the embedding system obtained from the liftings $(Z',N')$. Let $C_{(Y,M)/(W_{n},triv)}$
be the crystalline complex associated to the composite $(Z',N')\to(W_{n},triv)$.
Define \[
\tilde{C}_{Y}:=C_{(Y,M)/(W_{n},triv)}\otimes_{W_{n}<u>}W_{n}.\]
On the other hand, let $Z''=Z'\times_{\mbox{Spec }W_{n}[u]}\mbox{Spec }W_{n}<u>$
be endowed with $N''$ the inverse image of the log structure $N'$.
Let $\mathcal{L}$ be the log structure on $\mbox{Spec }W_{n}<u>$
obtained by taking the inverse image of (the log structure associated
to) $\mathbb{N}$ on $\mbox{Spec }W_{n}[u]$. Then $(Z'',N'')$ gives
rise to an embedding system for \[
(Y,M)\to(\mbox{Spec }W_{n}<u>,\mathcal{L}),\]
with crystalline complex $C_{(Y,M)/(\mbox{Spec }W_{n}<u>,\mathcal{L})}$.
Define \[
C_{Y}:=C_{(Y,M)/(\mbox{Spec }W_{n}<u>,\mathcal{L})}\otimes_{W_{n}<u>}W_{n}.\]
Note that $C_{Y}$ is the crystalline complex $C_{(Y,M)/(W_{n},\mathbb{N})}$
with respect to the embedding system obtained from $(Z'\times_{\mbox{Spec }W_{n}[u]}\mbox{Spec}W_{n},N''')$.
As in Section 3.6 of \cite{HK}, we have an exact sequence of complexes
\begin{equation}
0\to C_{Y}[-1]\to\tilde{C}_{Y}\to C_{Y}\to0,\label{eq:exact sequence}\end{equation}
where the second arrow is $\wedge\frac{du}{u}$ and the third arrow
is the canonical projection. The monodromy operator on the crystalline
cohomology of $(Y,M)$ is induced by the connecting homomorphism of
this exact sequence. 
\begin{lem}
\label{lem:independence}Let $C_{Z}^{\cdot}$ be either one of the
complexes $\tilde{\tilde{C}}_{Y}^{\cdot}$, $\tilde{C}_{Y}^{\cdot}$
or $C_{Y}^{\cdot}$ obtained with respect to a lifting $(Z,\tilde{N})$
of some cover $U\to Y$. In the derived category, $C_{Z}^{\cdot}$
is independent of the choice of lifting $(Z,\tilde{N})$. \end{lem}
\begin{proof}
We may work etale locally on $Y$, in which case we have to show that
for any two liftings $(Z_{1},\tilde{N}_{1})$ and $(Z_{2},\tilde{N}_{2})$
we have a canonical quasi-isomorphism between the corresponding complexes
and moreover, that these quasi-isomorphisms satisfy the obvious cocycle
condition for three different liftings. 

First, we show that the complexes corresponding to $(Z_{1},\tilde{N}_{1})$
and $(Z_{2},\tilde{N}_{2})$ are quasi-isomorphic. We may assume that
$i_{i}:(Y,\tilde{M})\to(Z_{i},\tilde{N}_{i})$ is an exact closed
immersion for $i=1,2$. Let $i_{12}:(Y,\tilde{M})\to(Z_{1}\times_{W_{n}}Z_{2},\tilde{N}_{1\times2})$
be the diagonal immersion of $(Y,\tilde{M})$ into the fiber product
of $(Z_{1},\tilde{N}_{1})$ and $(Z_{2},\tilde{N}_{2})$ as fs log
schemes over $(W_{n},triv)$. Let $(Z_{12,}\tilde{N}_{12})$ be a
log scheme such that etale locally on $Y$ we have a factorization
of $i_{12}$ \[
(Y,\tilde{M})\stackrel{f}{\to}(Z_{12},\tilde{N}_{12})\stackrel{g}{\to}(Z_{1}\times Z_{2},\tilde{N}_{1\times2}),\]
with $g$ log etale and $f$ an exact closed immersion. This factorization
is possible by Lemma 4.10 of \cite{K}. Let $D_{i}$ be the $PD$-envelope
of $Y$ in $Z_{i}$ (again, for $i=1,2$ or $12$). (Since we have
exact closed immersions, the logarithmic $PD$-envelope coincides
with the usual $PD$-envelope in these cases.) It suffices to show
that the canonical map \begin{equation}
\omega_{(Z_{1},\tilde{N}_{1})/W_{n},triv}^{\cdot}\otimes_{\mathcal{O}_{Z_{1}}}\mathcal{O}_{D_{1}}\to\omega_{(Z_{12},\tilde{N}_{12})/W_{n},triv}^{\cdot}\otimes_{\mathcal{O}_{Z_{12}}}\mathcal{O}_{D_{12}}\label{quasi}\end{equation}
is a quasi-isomorphism. This follows from paragraph 2.21 of \cite{HK}.
For completeness, we sketch the proof here. Let $p_{1}:(Z_{12},N_{12})\to(Z_{1},N_{1})$
be the log smooth map induced by projection onto the first factor.
For any geometric point $\bar{y}$ of $Y$, the stalks at $\bar{y}$
of $N_{12}$ and $p_{1}^{*}N_{1}$ coincide, so by replacing $(Z_{12},N_{12})$
with an etale neighborhood of $\bar{y}\to Z_{12}$, we may assume
that $N_{12}=p_{1}^{*}N_{1}$. Then the map $p_{1}:Z_{12}\to Z_{1}$
is smooth in the usual sense. Since the problem is etale local on
$Y$, we may assume that $Z_{12}\simeq Z_{1}\otimes_{W_{n}}W_{n}[t_{1},\dots,t_{r}]$
for some positive integer $r$ and such that $Y$ is contained in
the closed subscheme of $Z_{12}$ defined by $t_{1}=\dots=t_{r}=0$.
As in Proposition 6.5 of \cite{K}, we also have $\mathcal{O}_{D_{12}}\simeq\mathcal{O}_{D_{1}}<t_{1},\dots,t_{r}>$,
the PD-polynomial ring over $\mathcal{O}_{D_{1}}$ in $r$ variables.
The quasi-isomorphism (\ref{quasi}) is reduced then to the standard
quasi-isomorphism \[
W_{n}\to\Omega_{W_{n}[t_{1},\dots,t_{r}]}\otimes_{W_{n}[t_{1},\dots,t_{r}]}W_{n}<t_{1},\dots,t_{r}>.\]

The quasi-isomorphism \ref{quasi} commutes with $\otimes_{W_{n}<\tau,\sigma>}W_{n}$
so it induces a quasi-isomorphism \[
\tilde{\tilde{C}}_{Z_{1}}^{\cdot}\stackrel{\sim}{\to}\tilde{\tilde{C}}_{Z_{12}}^{\cdot}.\]
Now consider the morphism $Z'_{12}\to Z_{1}'$ obtained by pulling
back $Z_{12}\to Z_{1}$ along $G$. We claim that the canonical morphisms
$\tilde{C}_{Z_{12}}^{\cdot}\to\tilde{C}_{Z_{1}}^{\cdot}$ and $C_{Z_{12}}^{\cdot}\to C_{Z_{1}}^{\cdot}$
are quasi-isomorphisms as well. This is proved in the same way as
in the case of $\tilde{\tilde{C}}$ (for $C_{Z_{12}}^{\cdot}\to C_{Z_{1}}^{\cdot}$
it amounts to proving that the logarithmic de Rham-Witt complex is
independent of the choice of embedding system). The quasi-isomorphisms
are also compatible with the canonical maps $\tilde{\tilde{C}}_{Z}^{\cdot}\to\tilde{C}_{Z}^{\cdot}\to C_{Z}^{\cdot}$. 

Note that the above result also implies that in the derived category,
$C^{\cdot}$ commutes with etale base change. Indeed, if $Y_{2}/Y_{1}$
is etale and $(Z_{1},\tilde{N}_{1})$ is a lifting for $(Y_{1},\tilde{M})\to(\mbox{Spec }W_{n}[\tau,\sigma],\mathbb{N}^{2})$
then by \cite{EGA IV} 18.1.1 we can find, locally on $Y_{2}$, an
etale morphism $Z_{2}\to Z_{1}$ such that the following diagram is
cartesian\[
\xymatrix{Y_{2}\ar[r]\ar[d] & Z_{2}\ar[d]\\
Y_{1}\ar[r] & Z_{1}}
.\]
We take $\tilde{N}_{2}$ on $Z_{2}$ to be the inverse image of $\tilde{N}_{1}.$
Then $(Z_{2},\tilde{N}_{2})$ is a lifting for $(Y_{2},\tilde{M})\to(\mbox{Spec }W_{n}[\tau,\sigma],\mathbb{N}^{2})$
and, since log differentials commute with etale base change (Prop.
3.12 of \cite{K}), $C_{(Z_{2})}^{\cdot}$ on $Y_{2}$ is just the
pullback of $C_{(Z_{2})}^{\cdot}$ on $Y_{1}$. 

We are left with veryfing the cocycle condition. The canonical quasi-isomorphism
$\gamma_{12}:C_{Z_{1}}^{\cdot}\stackrel{\sim}{\to}C_{Z_{2}}^{\cdot}$
factors through $C_{Z_{1}\times Z_{2}}^{\cdot}$, since by construction
$Z_{12}$ is log etale over $Z_{1}\times Z_{2}$ and so we have a
quasi-isomorphism $C_{Z_{1}\times Z_{2}}^{\cdot}\stackrel{\sim}{\to}C_{Z_{12}}^{\cdot}$.
Let $(Z_{3},\tilde{N}_{3})$ be another lifting. Then we have the
following commutative diagram of complexes:\[
\xymatrix{ & C_{Z_{1}\times Z_{2}\times Z_{3}}^{\cdot}\\
C_{Z_{1}\times Z_{2}}^{\cdot}\ar[ru] & C_{Z_{1}\times Z_{3}}^{\cdot}\ar[u] & C_{Z_{2}\times Z_{3}}^{\cdot}\ar[lu]\\
C_{Z_{1}}^{\cdot}\ar[r]^{\gamma_{12}}\ar[u]\ar[ru] & C_{Z_{2}}^{\cdot}\ar[r]^{\gamma_{23}}\ar[ru]\ar[lu] & C_{Z_{3}}^{\cdot}\ar[u]\ar[lu]}
,\]
where all the maps are quasi-isomorphisms. This proves the cocycle
condition. \end{proof}
\begin{cor}
The following sheaves on $Y$ are well-defined and commute with etale
base change:\[
W_{n}\tilde{\tilde{\omega}}_{Y}^{q}:=\mathcal{H}^{q}\left(\tilde{\tilde{C}}_{Y}^{\cdot}\right),W_{n}\tilde{\omega}_{Y}^{q}:=\mathcal{H}^{q}\left(\tilde{C}_{Y}^{\cdot}\right)\mbox{ and }W_{n}\omega_{Y}^{q}:=\mathcal{H}^{q}\left(C_{Y}^{\cdot}\right),\]
The sheaves $W_{n}\omega_{Y}^{q}$ make up the $q$-th terms of the
log de Rham-Witt complex associated to $(Y,M)$. We have canonical
morphisms of sheaves on $Y$:\[
W_{n}\tilde{\tilde{\omega}}_{Y}^{q}\to W_{n}\tilde{\omega}_{Y}^{q}\to W_{n}\omega_{Y}^{q}.\]

\end{cor}
In order to understand the monodromy $N$, we will study the short
exact sequence of complexes\[
0\to W_{n}\omega_{Y}^{\cdot}[-1]\to W_{n}\tilde{\omega}_{Y}^{\cdot}\to W_{n}\omega_{Y}^{\cdot}\to0,\]
which we obtain below from the short exact sequence (\ref{eq:exact sequence}).
In Section 4 we will construct a resolution of this short exact sequence
in terms of some subquotients of $W_{n}\tilde{\tilde{\omega}}_{Y}^{\cdot}$.
For now, since these complexes are independent of the choice of lifting,
we will fix a specific kind of lifting of $(Y,\tilde{M})$ over $(W[\tau,\sigma],\mathbb{N}^{2})$,
which we call \emph{admissible liftings,} following the terminology
used in \cite{H} and \cite{Mo}. Since $Y$ is locally etale over
\[
Y_{r,s,m}=\mbox{Spec }k[X_{1},\dots,X_{n},Y_{1},\dots,Y_{n},Z_{1},\dots Z_{m}]/(X_{1}\cdot\dots\cdot X_{r},Y_{1}\cdot\dots\cdot Y_{s}),\]
we consider the lifting \[
Z_{r,s,m}=\mbox{Spec }W[X_{1},\dots,X_{n},Y_{1},\dots,Y_{n},Z_{1},\dots Z_{m},\tau,\sigma]/(X_{1}\cdot\dots\cdot X_{r}-\tau,Y_{1}\cdot\dots\cdot Y_{s}-\sigma).\]
of $(Y_{r,s,m},\mathbb{N}^{r}\oplus\mathbb{N}^{s})/(W[\tau.\sigma],\mathbb{N}^{2}).$
The log structure on $Z_{r,s,m}$ is also induced from $\mathbb{N}^{r}\oplus\mathbb{N}^{s}$
(with the obvious structure map sending $\mathbb{N}^{r}$ to products
of the $X_{i}$ and $\mathbb{N}^{s}$ to products of the $Y_{j}$).
We let $Z/Z_{r,s,m}$ to be etale and such that the diagram\[
\xymatrix{(Y,\tilde{M})\ar[r]\ar[d] & (Z,\tilde{N})\ar[d]\\
(Y_{r,s,m},\mathbb{N}^{r}\oplus\mathbb{N}^{s})\ar[r] & (Z_{r,s,m},\mathbb{N}^{r}\oplus\mathbb{N}^{s})}
,\]
is Cartesian, with the log structures on top obtained by pullback
from the ones on the bottom. Then locally on $Y$, the complexes $W_{n}\tilde{\tilde{\omega}}_{Y}^{\cdot}$,
$W_{n}\tilde{\omega}_{Y}^{\cdot}$ and $W_{n}\omega_{Y}^{\cdot}$
are just pullbacks of the corresponding complexes on $Y_{r,s,m}$
with respect to the lifting $(Z_{r,s,m},\mathbb{N}^{r}\oplus\mathbb{N}^{s})$.
Note that admissible liftings exist locally on $Y$. 

Now we will explain the relationships between $\tilde{\tilde{C}}_{Y}^{\cdot}$,
$\tilde{C}_{Y}^{\cdot}$ and $C_{Y}^{\cdot}$. First, note that we
have the functoriality map $G^{*}\omega_{(Z,\tilde{N})/(W_{n},triv)}\to\omega_{(Z',N')/(W_{n},triv)},$
which induces a canonical map \[
C_{(Y,\tilde{M})/(W_{n},triv)}^{\cdot}\otimes_{W_{n}<\tau,\sigma>}W_{n}<u>\to C_{(Y,M)/(W_{n},triv)}^{\cdot},\]
which in turn induces a canonical map $\tilde{\tilde{C}}_{Y}^{\cdot}\to\tilde{C}_{Y}^{\cdot}$.
By composition, we also get a map $\tilde{\tilde{C}}_{Y}^{\cdot}\to C_{Y}^{\cdot}$.
We claim that we can identify $\tilde{C}_{Y}^{\cdot}$ with $\tilde{\tilde{C}}_{Y}^{\cdot}/\left(\frac{d\tau}{\tau}-\frac{d\sigma}{\sigma}\right)\wedge\tilde{\tilde{C}}_{Y}^{\cdot}$
and $C_{Y}^{\cdot}$ with $\tilde{\tilde{C}}_{Y}^{\cdot}/\left(\frac{d\tau}{\tau}\wedge\tilde{\tilde{C}}_{Y}^{\cdot}+\frac{d\sigma}{\sigma}\wedge\tilde{\tilde{C}}_{Y}^{\cdot}\right).$
We explain this in the case of $\tilde{C}_{Y}^{\cdot}$. 
\begin{lem}
\label{lem:Tilde in tems of tilde^2}We have an isomorphism \[
\tilde{\tilde{C}}_{Y}^{\cdot}/\left(\frac{d\tau}{\tau}-\frac{d\sigma}{\sigma}\right)\wedge\tilde{\tilde{C}}_{Y}^{\cdot-1}\stackrel{\sim}{\to}\tilde{C}_{Y}^{\cdot}.\]
\end{lem}
\begin{proof}
Let $(Z,\tilde{N})$ be an admissible lifting of $(Y,\tilde{M})$
over $(\mbox{Spec }W_{n}[\tau,\sigma],\mathbb{N}^{2})$. Let $(D,\tilde{M}_{D})$
be the divided power envelope of $(Y,\tilde{M})$ in $(Z,\tilde{N})$.
Note that the kernel of the map $\mathcal{O}_{D}\to\mathcal{O}_{Y}$
is generated by $\tau^{[n]}$ and $\sigma^{[n]}$. The divided power
envelope $(D',M'_{D})$ of $(Y,M)$ in $(Z',N')$ satisfies the following
property:\[
\mathcal{O}_{D'}\simeq\mathcal{O}_{D}\otimes_{W_{n}<\tau,\sigma>}W_{n}<u>,\]
where the map $W_{n}<\tau,\sigma>\to W_{n}<u>$ is $\tau^{[n]},\sigma^{[n]}\mapsto u^{[n]}$.
The complexes $\tilde{\tilde{C}}_{Y}^{\cdot}$ and $\tilde{C}_{Y}^{\cdot}$
are defined as follows: \[
\tilde{\tilde{C}}_{Y}^{\cdot}:=(\omega_{(Z,\tilde{N})/(W_{n},triv)}^{\cdot}\otimes_{\mathcal{O}_{Z}}\mathcal{O}_{D})\otimes_{W_{n}<\tau,\sigma>}W_{n}<u>\otimes_{W_{n}<u>}W_{n}=\]
\[
=(\omega_{(Z,\tilde{N})/(W_{n},triv)}^{\cdot}\otimes_{W_{n}[\tau,\sigma]}W_{n}[u])\otimes_{\mathcal{O}_{Z'}}\mathcal{O}_{D'}\otimes_{W_{n}<u>}W_{n}\]
 and \[
\tilde{C}_{Y}^{\cdot}=(\omega_{Z',N'/W_{n},triv}^{\cdot})\otimes_{\mathcal{O}_{Z'}}\mathcal{O}_{D'}\otimes_{W_{n}<u>}W_{n}.\]
Note that since we've chosen an admissible lifting $(Z',N')$ has
$Z\times_{W_{n}[\tau,\sigma]}W_{n}[u]$ as its underlying scheme because
$\tilde{N}\oplus_{\mathbb{N}^{2}}\mathbb{N}$ is already fine and
saturated. It is enough to show that the sequence \begin{equation}
\omega_{(Z,\tilde{N})/(W_{n},triv)}^{\cdot-1}\otimes_{W_{n}[\tau,\sigma]}W_{n}[u]\stackrel{\wedge\left(\frac{d\tau}{\tau}-\frac{d\sigma}{\sigma}\right)}{\longrightarrow}\omega_{(Z,\tilde{N})/(W_{n},triv)}^{\cdot}\otimes_{W_{n}[\tau,\sigma]}W_{n}[u]\to\omega_{(Z',N')/(W_{n},triv)}^{\cdot}\to0\label{eq:tilde exact}\end{equation}
is exact, where the second map is induced by functoriality. We denote
by $G^{*}$ the pullback along $\mbox{Spec }W_{n}[u]\to\mbox{Spec }W_{n}[\tau,\sigma]$
or along $Z'\to Z$. By proposition 3.12 of \cite{K}, we have the
following diagram of exact sequences of sheaves on $Z'$ \[
\xymatrix{ & G^{*}\omega_{(\mbox{Spec }W_{n}[\tau,\sigma],\mathbb{N}^{2})/(W_{n},triv)}^{1}\otimes_{W_{n}[u]}\mathcal{O}_{Z'}\ar[d]\ar[r] & G^{*}\omega_{(Z,\tilde{N})/(W_{n},triv)}^{1}\ar[d]\ar[r] & G^{*}\omega_{(Z,\tilde{N})/(\mbox{Spec }W_{n}[\tau,\sigma],\mathbb{N}^{2})}^{1}\ar[d]\ar[r] & 0\\
0\ar[r] & \omega_{(\mbox{Spec }W_{n}[u],\mathbb{N})/(W_{n},triv)}^{1}\otimes_{W_{n}[u]}\mathcal{O}_{Z'}\ar[r] & \omega_{(Z',N')/(W_{n},triv)}^{1}\ar[r] & \omega_{(Z',N')/(\mbox{Spec }W_{n}[u],\mathbb{N})\ar[r]}^{1} & 0}
.\]
The rightmost vertical arrow is an isomorphism, since $(Z',N')$ was
obtained by pullback from $(Z,\tilde{N})$. In order to show that
the middle vertical arrow is a surjection, it is enough to check that
$\frac{du}{u}$ is in its image, but both $\frac{d\sigma}{\sigma}$
and $\frac{d\tau}{\tau}$ map to $\frac{du}{u}$ . We also see similarly
that the kernel of the middle vertical arrow is generated by $\frac{d\tau}{\tau}-\frac{d\sigma}{\sigma}$.
The exactness of (\ref{eq:tilde exact}) follows.\end{proof}
\begin{cor}
\label{cor:C in terms of tilde ^2}We have an isomorphism \[
\tilde{\tilde{C}}_{Y}^{\cdot}/\left(\frac{d\tau}{\tau}\wedge\tilde{\tilde{C}}_{Y}^{\cdot-1}+\frac{d\sigma}{\sigma}\wedge\tilde{\tilde{C}}_{Y}^{\cdot-1}\right)\stackrel{\sim}{\to}C_{Y}^{\cdot}.\]
\end{cor}
\begin{proof}
This follows from the exact sequence (\ref{eq:exact sequence}) and
the Lemma \ref{lem:Tilde in tems of tilde^2}. \end{proof}
\begin{lem}
\label{global sections}The sections $\frac{d\tau}{\tau}$ and $\frac{d\sigma}{\sigma}\in W_{n}\tilde{\tilde{\omega}}_{Y}^{1}$
are global sections, independent of the choice of admissible lifting.
The same holds for $\frac{du}{u}\in W_{n}\tilde{\omega}_{Y}^{1}.$ \end{lem}
\begin{proof}
We will explain the proof only for $\frac{d\tau}{\tau}$ since the
same proof also works for $\frac{d\sigma}{\sigma}$ and $\frac{du}{u}$.
We use basically the same argument as for Lemma 3.4 of \cite{Mo},
part 3. We consider two admissible liftings of $(Y,\tilde{M})$, $(Z_{1},\tilde{N}_{1})$
and $(Z_{2},\tilde{N}_{2})$ and we let $(Z_{12},\tilde{N}_{12})$
be defined as in Lemma \ref{lem:independence}. It is enough to show
that locally on $Y$ \[
\frac{d\tau}{\tau}\in\omega_{(Z_{1},\tilde{N}_{1})/(W_{n},triv)}^{1}\otimes_{\mathcal{O}_{Z_{1}}}\mathcal{O}_{D_{1}}\]
 and \[
\frac{d\tau'}{\tau'}\in\omega_{(Z_{2},\tilde{N}_{2})/(W_{n},triv)}^{1}\otimes_{\mathcal{O}_{Z_{2}}}\mathcal{O}_{D_{2}}\]
have the same image in $\mathcal{H}^{1}(\omega_{(Z_{12},\tilde{N}_{12})/(W_{n},triv)}^{\cdot}\otimes_{\mathcal{O}_{Z_{12}}}\mathcal{O}_{D_{12}}).$ 

Note that $\frac{d\tau}{\tau}\in\tilde{N}_{1}$ and $\frac{d\tau'}{d\tau'}\in\tilde{N}_{2}$
have the same image in $\tilde{M}$. This is because locally on $Y$
we have commutative diagrams \[
\xymatrix{(Y,\tilde{M})\ar[d]\ar[r] & (Z_{i},\tilde{N}_{i})\ar[d]\\
(k,\mathbb{N}^{2})\ar[r] & (W_{n}[\tau,\sigma],\mathbb{N}^{2})}
\]
for $i=1,2$, so both $\frac{d\tau}{\tau}$ and $\frac{d\tau'}{\tau'}$
map to the image of $(1,0)\in\mathbb{N}^{2}$ in $\tilde{M}$. By
the construction of $(Z_{12},\tilde{N}_{12})$, (see the proof of
Prop. 4.10 of \cite{K}) we know that $\frac{d\tau}{\tau}-\frac{d\tau'}{\tau'}=m\in\tilde{N}_{12}$.
Moreover, if $\alpha_{12}:N_{12}\to\mathcal{O}_{Z_{12}}$ is the map
defining the log structure of $Z_{12}$ then $m$ maps to $0\in\tilde{M}$,
so $v=\alpha_{12}(m)\in\mathcal{O}_{Z_{12}}^{\times}$ maps to $1\in\mathcal{O}_{Y}$.
Therefore, \[
\frac{d\tau}{\tau}-\frac{d\tau'}{\tau'}=\frac{dv}{v},\]
for some $v\in\mathcal{O}_{D_{12}}$ for which $W_{n}<v-1>\subseteq\mathcal{O}_{D_{12}}$.
But then we see that $\frac{dv}{v}\in d(W_{n}<v-1>)$ using the fact
that the power series expansion of $\log(v)$ around $1$ belongs
to $W_{n}<v-1>$. Therefore, $\frac{d\tau}{\tau}-\frac{d\tau'}{\tau'}$
is exact and the lemma follows. 
\end{proof}
As in the classical case (\cite{IR,HK}), we can define operators
$F:W_{n+1}\tilde{\tilde{\omega}}^{q}\to W_{n}\tilde{\tilde{\omega}}^{q}$,
$V:W_{n}\tilde{\tilde{\omega}}^{q}\to W_{n+1}\tilde{\tilde{\omega}}^{q}$
and the differential $d:W_{n}\tilde{\tilde{\omega}}^{q}\to W_{n}\tilde{\tilde{\omega}}^{q+1}$,
which satisfy \[
d^{2}=0,FV=VF=p,dF=pFd,Vd=pdV\mbox{ and }FdV=V.\]
Indeed, fix local liftings $(Z_{n},\tilde{N}_{n})$ of $(Y,\tilde{M})\to(\mbox{Spec }W_{n}[\tau,\sigma],\mathbb{N}^{2})$
and denote the crystalline complex $\tilde{\tilde{C}}_{Z_{n}}^{\cdot}$
by $\tilde{\tilde{C}}_{n}^{\cdot}$. We can see that $\tilde{\tilde{C}}_{n}^{\cdot}$
is flat over $W_{n}$ in the same way as in Lemma 2.22 of \cite{HK}
(using an admissible lifting) and we have \[
\tilde{\tilde{C}}_{n}^{\cdot}\otimes_{\mathbb{Z}/p^{n}\mathbb{Z}}\mathbb{Z}/p^{m}\mathbb{Z}\stackrel{\sim}{\to}\tilde{\tilde{C}}_{m}^{\cdot}\]
for $m\leq n$. We let $F:W_{n+1}\tilde{\tilde{\omega}}^{\cdot}\to W_{n}\tilde{\tilde{\omega}}^{\cdot}$
be the map induced by $\tilde{\tilde{C}}_{n+1}^{\cdot}\to\tilde{\tilde{C}}_{n}^{\cdot}$,
$V:W_{n}\tilde{\tilde{\omega}}^{\cdot}\to W_{n+1}\tilde{\tilde{\omega}}^{\cdot}$
be the map induced by $p:\tilde{\tilde{C}}_{n}^{\cdot}\to\tilde{\tilde{C}}_{n+1}^{\cdot}$.
We define $d$ to be the connecting homomorphism in the exact sequence
of cohomology sheaves associated to the exact sequence of crystalline
complexes\[
0\to\tilde{\tilde{C}}_{n}^{\cdot}\stackrel{p^{n}}{\to}\tilde{\tilde{C}}_{2n}^{\cdot}\to\tilde{\tilde{C}}_{n}^{\cdot}\to0.\]
The same operators can be defined for $W_{\cdot}\tilde{\omega}_{Y}^{\cdot}$
and $W_{\cdot}\omega_{Y}^{\cdot}$.
\begin{lem}
\label{lem:Cartier}Let $n=1$. Locally, fix an admissible lifting
$(Z,\tilde{N})$ as above. Let $Fr$ be the relative Frobenius of
$Y/k$. We have Cartier isomorphisms\[
C^{-1}:\omega_{Y}^{q}\stackrel{\sim}{\to}\mathcal{H}^{q}(Fr_{*}\omega_{Y}^{\cdot}),\]
\[
\tilde{C}^{-1}:\omega_{(Z',N')/(k,triv)}^{q}\otimes_{k[u]}k\stackrel{\sim}{\to}\mathcal{H}^{q}(Fr_{*}(\omega_{(Z',N')/k,triv}^{\cdot}\otimes_{k[u]}k))\]
and \[
\tilde{\tilde{C}}^{-1}:\omega_{(Z,\tilde{N})/(k,triv)}^{q}\otimes_{k[t,s]}k\stackrel{\sim}{\to}\mathcal{H}^{q}(Fr_{*}(\omega_{(Z,\tilde{N})/k,triv}^{\cdot}\otimes_{k[t,s]}k)).\]
\end{lem}
\begin{proof}
Note that $(Y,M)/(\mbox{Spec }k,\mathbb{N})$ is log smooth of Cartier
type. The Cartier isomorphism for $W_{1}\omega_{Y}^{q}$ is then defined
in section 2.12 of \cite{HK}. Similarly, $(Z',N')/(\mbox{Spec }k,triv)$
and $(Z,\tilde{N})/(\mbox{Spec }k,triv)$ are log smooth and of Cartier
type. Thus, the morphisms $\tilde{C}^{-1}$ and $\tilde{\tilde{C}}^{-1}$
for $\tilde{C}_{Y}^{q}$ and $\tilde{\tilde{C}}_{Y}^{q}$ are induced
from the Cartier isomorphisms for these schemes. 

Since we are working locally on $Y,$ we may assume that $Y=Y_{1}\times_{k}Y_{2}$
and that the lifting $Z=Z_{1}\times Z_{2}$, where $Z_{1},Z_{2}$
are smooth over $k$ and $Y_{i}$ is a reduced normal crossings divisor
in $Z_{i}$. Let $\mathcal{I}_{i}$ be the ideal defining $Y_{i}\times_{k}Z_{3-i}$
in $Z$ for $i=1,2$. To check that $\tilde{\tilde{C}}^{-1}$ is an
isomorphism, we use the following commutative diagram of exact sequences:\[
\xymatrix{\omega_{(Z,\tilde{N})/k}^{q}\otimes\mathcal{I}_{1}\mathcal{I}_{2}\ar[r]\ar[d] & \omega_{(Z,\tilde{N})/k}^{q}\otimes\mathcal{I}_{1}\oplus\omega_{(Z,\tilde{N})/k}^{q}\otimes\mathcal{I}_{2}\ar[r]\ar[d] & \omega_{(Z,\tilde{N})/k}^{q}\ar[r]\ar[d] & \tilde{\tilde{C}}_{Y}^{q}\ar[r]\ar[d] & 0\\
\mathcal{H}^{q}(Fr_{*}\omega_{(Z,\tilde{N})/k}^{\cdot}\otimes\mathcal{I}_{1}\mathcal{I}_{2})\ar[r] & \mathcal{H}^{q}(Fr_{*}(\omega_{(Z,\tilde{N})/k}^{\cdot}\otimes\mathcal{I}_{1}\oplus\omega_{(Z,\tilde{N})/k}^{\cdot}\otimes\mathcal{I}_{2}))\ar[r] & \mathcal{H}^{q}(Fr_{*}\omega_{(Z,\tilde{N})/k}^{\cdot})\ar[r] & \mathcal{H}^{q}(F_{*}\tilde{\tilde{C}}_{Y}^{\cdot})\ar[r] & 0.}
\]
The complex $\omega_{(Z,\tilde{M})/k,triv}^{\cdot}$ is the same as
$\Omega_{Z_{1}/k}^{\cdot}(\log Y_{1})\otimes_{k}\Omega_{Z_{2}/k}^{\cdot}(\log Y_{2}),$
so it does satisfy a Cartier isomorphism, by 4.2.1.1 of \cite{DI}.
Similarly, the complexes on its left are (sums of) products of complexes
of the form $\Omega_{Z_{i}/k}^{\cdot}(\pm\log Y_{i})$ for $i=1,2$,
which also satisfy a Cartier isomorphism, by 4.2.1.3 of \cite{DI}.
Therefore, the first three vertical arrows are isomorphisms. Once
we know the exactness of the top and bottom sequence we can also deduce
that the rightmost vertical arrow is an isomorphism. The exactness
of the top row follows from the definition of $\tilde{\tilde{C}}_{Y}^{q}$.

The exactness of the bottom row follows from the cohomology long exact
sequence associated to short exact sequences from the top row combined
with the Cartier isomorphisms for the first three arrows which tell
us that the coboundary morphisms of these short exact sequences are
all $0$. Indeed, if we let $\bar{\omega}_{(Z,\tilde{N})}^{\cdot}$
be the complex obtained by completing the inclusion of complexes \[
\omega_{(Z,\tilde{N})/k}^{\cdot}\otimes\mathcal{I}_{1}\mathcal{I}_{2}\to\omega_{(Z,\tilde{N})/k}^{\cdot}\otimes\mathcal{I}_{1}\oplus\omega_{(Z,\tilde{N})/k}^{\cdot}\otimes\mathcal{I}_{2}\]
to a distinguished triangle, then we get a long exact sequence\[
\dots\to\mathcal{H}^{q}(\omega_{(Z,\tilde{N})/k}^{\cdot}\otimes\mathcal{I}_{1}\mathcal{I}_{2})\to\mathcal{H}^{q}(\omega_{(Z,\tilde{N})/k}^{\cdot}\otimes\mathcal{I}_{1})\oplus\mathcal{H}^{q}(\omega_{(Z,\tilde{N})/k}^{\cdot}\otimes\mathcal{I}_{2})\to\mathcal{H}^{q}(\bar{\omega}_{(Z,\tilde{N})}^{\cdot})\to\dots.\]
From the Cartier isomorphism for and, we deduce that \[
\mathcal{H}^{q}(\omega_{(Z,\tilde{N})/k}^{\cdot}\otimes\mathcal{I}_{1}\mathcal{I}_{2})\hookrightarrow\mathcal{H}^{q}(\omega_{(Z,\tilde{N})/k}^{\cdot}\otimes\mathcal{I}_{1})\oplus\mathcal{H}^{q}(\omega_{(Z,\tilde{N})/k}^{\cdot}\otimes\mathcal{I}_{2}),\]
so the coboundaries of the long exact sequence are all $0$. By continuing
this argument, we deduce the exactness of the entire bottom row and
this proves that $\tilde{\tilde{C}}^{-1}$ is an isomorphism. 

Now we prove that $\tilde{C}^{-1}$ is an isomorphism. We will show
that $\tilde{C}^{-1}$ is an insomorphism in degree $q$ as well.
From the short exact sequence (\ref{eq:exact sequence}), we get the
following commutative diagram with exact rows:\[
\xymatrix{0\ar[r] & C_{Y}^{q-1}\ar[r]\ar[d] & \tilde{C}_{Y}^{q}\ar[r]\ar[d] & C_{Y}^{q}\ar[r]\ar[d] & 0\\
0\ar[r] & \mathcal{H}^{q-1}(Fr_{*}C_{Y}^{\cdot})\ar[r] & \mathcal{H}^{q}(Fr_{*}\tilde{C}_{Y}^{\cdot})\ar[r] & \mathcal{H}^{q}(Fr_{*}C_{Y}^{\cdot})\ar[r] & 0}
.\]
To see that the bottom row is exact, we have to check that in the
long exact cohomology sequence associated to the top row the coboundaries
are all $0$, which is equivalent to showing surjectivity of $\mathcal{H}^{q}(Fr_{*}\tilde{C}_{Y}^{\cdot})\to\mathcal{H}^{q}(Fr_{*}C_{Y}^{\cdot})$.
However, by the top row and the Cartier isomorphism $C^{-1}$, the
composite\[
\tilde{C}_{Y}^{q}\to C_{Y}^{q}\to\mathcal{H}^{q}(Fr_{*}C_{Y}^{\cdot})\]
is surjective, so the desired map is surjective as well. Now we have
a map of short exact sequences, where the left and right vertical
maps are isomorphisms, so the middle one must be as well. 
\end{proof}
Using the Cartier isomorphisms, we can define canonical projections
$\pi:W_{n+1}\tilde{\tilde{\omega}}_{Y}^{\cdot}\to W_{n}\tilde{\tilde{\omega}}_{Y}^{\cdot}$.
The construction works in the same way for $W_{n}\tilde{\omega}_{Y}^{\cdot}$.
The definition of $\pi$ for $W_{n}\omega_{Y}^{\cdot}$ can be found
in section 1 of \cite{H} in the semistable case and in section 4
of \cite{HK} in general. The constructions in \cite{H} and in \cite{HK}
are the same, although they are formulated slighlty differently. Our
construction follows that in section 1 of \cite{H}, by first defining
a map $\mathfrak{p}:W_{n}\tilde{\tilde{\omega}}_{Y}^{\cdot}\to W_{n+1}\tilde{\tilde{\omega}}_{Y}^{\cdot}$
and then showing that $\mathfrak{p}$ is injective and its image coincides
with the image of multiplication by $p$ on $W_{n+1}\tilde{\tilde{\omega}}_{Y}^{\cdot}$.
The projection $\pi$ will then be the unique map which makes the
following diagram commute:

\[
\xymatrix{W_{n}\tilde{\tilde{\omega}}_{Y}^{\cdot}\ar[d]^{\mathfrak{p}} & W_{n+1}\tilde{\tilde{\omega}}_{Y}^{\cdot}\ar[l]^{\pi}\ar[ld]^{p}\\
W_{n+1}\tilde{\tilde{\omega}}_{Y}^{\cdot}}
.\]
The map $\mathfrak{p}:W_{n}\tilde{\tilde{\omega}}_{Y}^{i}\to W_{n+1}\tilde{\tilde{\omega}}_{Y}^{i}$
is induced from $p^{-i+1}Fr^{*}:\tilde{\tilde{C}}_{Y}^{i}\to\tilde{\tilde{C}}_{Y}^{i}$,
where $Fr:(Z,\tilde{N})\to(Z,\tilde{N})$ is a lifting of the Frobenius
endomorphism of $(Z,\tilde{N})\times_{W}k$ such that $Fr^{*}(W[\tau,\sigma])\subset W[\tau,\sigma]$.
The injectivity of $\mathfrak{p}$ and the fact that its image coincides
with that of mulriplication by $p$ are deduced as in Section 2 of
\cite{H} (or as in Lemma 6.8 of \cite{Na}) from the Cartier isomorphism
and from the fact that $\tilde{\tilde{C}}_{Y}^{\cdot}$ is $W-$torsion-free
(when we take $\tilde{\tilde{C}}_{Y}^{\cdot}$ to be the crystalline
complex associated to an embedding system for $(Y,\tilde{M}$) over
$W$). 

Now we will consider a different interpretation of the monodromy operator
$N$. Taking the cohomology sheaves of the short exact sequence \[
0\to C_{Y}^{\cdot}[-1]\to\tilde{C}_{Y}^{\cdot}\to C_{Y}^{\cdot}\to0\]
 we get a long exact sequence of sheaves on $Y$ \[
\dots\to W_{n}\omega_{Y}^{q-1}\to W_{n}\tilde{\omega}_{Y}^{q}\to W_{n}\omega_{Y}^{q}\to\dots\]
whose coboundaries are actually all $0$. This can be checked as in
Lemma 1.4.3 of \cite{H}, since it suffices to see that the induced
map on cocycles $Z^{q}(\tilde{C}_{Y})\to Z^{q}(C_{Y})$ modulo $p^{n}$
is surjective and we can use the Cartier isomorphisms in Lemma \ref{lem:Cartier}
to give an explicit formula for cocycles modulo $p^{n}$. So we have
a short exact sequence of sheaves on $Y$ \begin{equation}
0\to W_{n}\omega_{Y}^{q-1}\to W_{n}\tilde{\omega}_{Y}^{q}\to W_{n}\omega_{Y}^{q}\to0,\label{eq: W-exact sequence}\end{equation}
which is compatible with operators $\pi,F,V$ and $d$. We have a
morphism of distinguished triangles in the derived category $D(Y_{et},W)$
of sheaves of $W$-modules on $Y$: \[
\xymatrix{C_{Y}^{\cdot}[-1]\ar[r]\ar[d] & \tilde{C}_{Y}^{\cdot}\ar[r]\ar[d] & C_{Y}^{\cdot}\ar[r]\ar[d] & C_{Y}^{\cdot}\ar[d]\\
W_{n}\omega_{Y}^{\cdot}[-1]\ar[r] & W_{n}\tilde{\omega}_{Y}^{\cdot}\ar[r] & W_{n}\omega_{Y}^{\cdot}\ar[r]. & W_{n}\omega_{Y}^{\cdot}}
\]
The left and right vertical maps are defined in the proof of Theorem
4.19 of \cite{HK} and the middle one can be defined in exactly the
same way. Note that the definition of the maps in Theorem 4.19 has
a gap which is corrected in Lemma 7.18 of \cite{Na}, namely checking
that they commute with the transition morphisms $\pi:W_{n+1}\omega_{Y}^{\cdot}\to W_{n}\omega_{Y}^{\cdot}$.
The fact that the middle map commutes with the transition morphisms
$\pi:W_{n+1}\tilde{\omega}_{Y}^{\cdot}\to W_{n}\tilde{\omega}_{Y}^{\cdot}$
can be checked in the same way as in Lemma 7.18 of \cite{Na}, using
the corresponding Cartier isomorphism to check that the complexes
$W_{n}\tilde{\omega}_{Y}^{\cdot}$ give rise to formal de Rham-Witt
complexes as in definition 6.1 of loc. cit. and thus applying Corollary
6.28 (8). We also need to check that that $\lim_{\leftarrow}W_{n}\tilde{\omega}_{Y}^{1}$
is torsion-free, but we can use the fact that this is known for $\lim_{\leftarrow}W_{n}\omega_{Y}^{1}$
and the exact sequence (\ref{eq: W-exact sequence}). The first and
third vertical maps are quasi-isomorphisms by theorem 4.19 of \cite{HK},
so we get an isomorphism of distinguished triangles. Thus, the exact
sequence (\ref{eq: W-exact sequence}) induces the monodromy operator
$N$ on cohomology. 

Assume that $Y$ has an admissible lifting $\underline{Z}$ over $(W[t,s],\mathbb{N}^{2})$
and set $Z=\underline{Z}\otimes_{W}k$. We consider a few more variations
on the de Rham Witt complex, which we will only define locally on
$Z$. Let $W_{n}\Omega_{Z}^{\cdot}$ be the de Rham Witt complex of
$Z$. Let \[
Y^{1}=\mathrm{Spec}\ k[X_{1},\dots,X_{n},Y_{1},\dots,Y_{n},Z_{1},\dots,Z_{m}]/X_{1}\cdot\dots\cdot X_{r}\]
and \[
Y^{2}=\mathrm{Spec}\ k[X_{1},\dots,X_{n},Y_{1},\dots,Y_{n},Z_{1},\dots,Z_{m}]/Y_{1}\cdot\dots\cdot Y_{s}.\]
Each $Y^{i}$ is a normal crossings divisor in $Z_{r,s,m}\times_{W}k$.
Let $\mathcal{D}_{n}^{i}$ be the structure sheaf of the divided power
envelope of $Y^{i}$ in $Z_{r,s,m}$ and $\mathcal{ID}_{n}^{i}=\ker(\mathcal{D}_{n}^{i}\to\mathcal{O}_{Y^{i}}).$
For $i=1,2$ let $W_{n}\Omega_{Z}^{\cdot}(-\log Y^{i})$ be the (pullback
to $Z$) of the {}``compact support'' version of de Rham Witt complex
of $Z_{r,s,m}$ with respect to $Y^{i}$. This complex was introduced
by Hyodo in section 1 of \cite{H} and it is defined by \[
W_{n}\Omega_{Z_{r,s,m}}^{q}(-\log Y^{i})=\mathcal{H}^{q}(\Omega_{Z/W_{n}}^{\cdot}(\log Y^{i})\otimes_{\mathcal{O}_{Z_{r,s,m}}}\mathcal{ID}_{n}^{i})\]
Let $W_{n}\Omega_{Z}^{\cdot}(-\log Y^{1}-\log Y^{2})$ be the pullback
from $Z_{r,s}$ to $Z$ of the complex defined by \[
W_{n}\Omega_{Z_{r,s,m}}^{q}(-\log Y^{1}-\log Y^{2}):=\mathcal{H}^{q}(\omega_{Z_{r,s,m},\mathbb{N}^{r}\oplus\mathbb{N}^{s}/W_{n}}^{\cdot}\otimes_{\mathcal{O}_{Z}}\mathcal{I}\mathcal{D}_{1}\mathcal{ID}_{2}).\]
This third complex is meant to approximate a product of complexes
of the form $W_{n}\Omega_{Z}(-\log Y)$. When $n=1,$ consider $Z^{1}=\mbox{Spec }k[X_{1},\dots,X_{n},t]/(X_{1}\cdot\dots\cdot X_{r}-t)$,
$Z^{2}=\mbox{Spec }k[Y_{1},\dots,Y_{n},u]/(Y_{1}\cdot\dots\cdot Y_{s}-u)$
and $Z^{3}=\mbox{Spec }k[Z_{1},\dots,Z_{m}]$. Then \begin{equation}
W_{1}\Omega_{Z_{r,s,m}}^{\cdot}(-\log Y^{1}-\log Y^{2})\simeq\Omega_{Z^{1}/k}^{\cdot}(-\log Y^{1})\otimes_{k}\Omega_{Z^{2}/k}^{\cdot}(-\log Y^{2})\otimes_{k}\Omega_{Z^{3}/k}^{\cdot}.\label{product formula}\end{equation}
All these also are endowed with operators $F,V$, differential $d$
and projection $\pi$, and they also satisfy a Cartier isomorphism. 
\begin{lem}
\label{tensor with R_n}Let $W_{n}\Omega^{\cdot}$ be either of the
complexes $W_{n}\Omega_{Z}^{\cdot},W_{n}\Omega_{Z}^{\cdot}(-\log Y^{i})$
for $i=1,2$ or $W_{n}\Omega_{Z}^{\cdot}(-\log Y^{1}-\log Y^{2})$.
Let \[
W\Omega^{\cdot}=\lim_{\leftarrow}W_{n}\Omega^{\cdot}.\]
Then $W\Omega^{\cdot}\otimes_{\mathbb{R}}^{L}\mathbb{R}_{n}=W_{n}\Omega^{\cdot}$. \end{lem}
\begin{proof}
For $n=1$, and $W_{n}\Omega_{Z}^{\cdot}$ and $W_{n}\Omega_{Z}^{\cdot}(-\log Y^{i})$
we have Cartier isomorphisms\[
W_{1}\Omega^{i}\stackrel{\sim}{\to}\mathcal{H}^{i}(F_{*}W_{1}\Omega^{\cdot}),\]
by result 4.2.1.3 in \cite{DI}. For $W_{n}\Omega_{Z}^{\cdot}(-\log Y^{1}-\log Y^{2})$
the Cartier isomorphism follows from the product formula (\ref{product formula})
and from the Cartier isomorphisms above. Let $\mathcal{Z}_{n}=\underline{Z}\times_{W}W_{n}$.
By abuse of notation, we write $\Omega_{Z_{n}}^{\cdot}$ for the complex
of sheaves of $W_{n}$-modules such that \[
W_{n}\Omega^{i}=\mathcal{H}^{i}(\Omega_{\mathcal{Z}_{n}}^{\cdot}).\]
In fact, we have complexes $\Omega_{\underline{Z}}^{\cdot},$ $\Omega_{\underline{Z}}^{\cdot}(-\log Y^{i})$
or $\Omega_{\underline{Z}}^{\cdot}(-\log Y^{1}-\log Y^{2})$) which
give the corresponding complexes $\Omega_{\mathcal{Z}_{n}}^{\cdot}$,
$\Omega_{\mathcal{Z}_{n}}^{\cdot}(-\log Y^{i})$ or $\Omega_{\mathcal{Z}_{n}}^{\cdot}(-\log Y^{1}-\log Y^{2})$)
when reduced modulo $p^{n}$. We also denote any of the initial complexes
over $W$ as $\Omega_{\underline{Z}}^{\cdot}$. Then there is an explicit
description of cocycles modulo $p^{n}$, which is given by \[
d^{-1}(p^{n}\Omega_{\underline{Z}}^{i+1})=\sum_{k=0}^{n}p^{k}f^{n-k}\Omega_{\underline{Z}}^{i}+\sum_{k=0}^{n-1}f^{k}d\Omega_{\underline{Z}}^{i-1},\]
where $f:\Omega_{\underline{Z}}^{i}\to\Omega_{\underline{Z}}^{i}$
is defined by $f=Fr/p^{i}$. This is the same as formula A from editorial
comment 11 in \cite{H} and is proven in the same way as in that paper
and in the same way as in the classical crystalline cohomology case
(see 0.2.3.13 of \cite{I}). 

As in the case of $W_{n}\omega_{Y}$, $W_{\cdot}\Omega^{\cdot}$ (and
$W\Omega^{\cdot})$ are endowed with a differential $d,$ operators
$F,V$ satisfying the usual relations and a canonical projection $\pi_{n}:W_{n+1}\Omega^{\cdot}\to W_{n}\Omega^{\cdot}$
such that $\mathfrak{p}\circ\pi_{n}$ coincides with multiplication
by $p$ on $W_{n+1}\Omega^{\cdot}$. 

We claim that the lemma follows from the Cartier isomorphism, from
the description of cocycles modulo $p^{n}$ in $\Omega_{\underline{Z}}^{\cdot}$
and from the formal properties of $W_{n}\Omega^{\cdot}$. The proof
is the same as for Lemma 1.3.3 of \cite{Mo}. We outline the argument
in order to show that it applies to our case as well. To prove the
desired result, we use the flat resolution of $\mathbb{R}_{n}$ as
an $\mathbb{R}$-module given by \[
0\to\mathbb{R}\stackrel{(F^{n},-F^{n}d)}{\longrightarrow}\mathbb{R}\oplus\mathbb{R}\stackrel{dV^{n}+V^{n}}{\longrightarrow}\mathbb{R}\to\mathbb{R}_{n}\to0\]
and it suffices by Corollary 1.3.3 of \cite{IR} to prove that the
sequence\[
0\to W\Omega^{i-1}\stackrel{(F^{n},-F^{n}d)}{\longrightarrow}W\Omega^{i-1}\oplus W\Omega^{i}\stackrel{dV^{n}+V^{n}}{\longrightarrow}W\Omega^{i}\to W_{n}\Omega^{i}\to0\]
is exact. The last map is the canonical projection $\pi:W\Omega^{i}\to W_{n}\Omega^{i}$. 

Exactness at the first term follows from the fact that multiplication
by $p$ (and hence also $F$) is injective on $W\Omega^{\cdot}$.
Indeed, multiplication by $p$ on $W_{n}\Omega^{\cdot}$ factors as
$\mathfrak{p}\circ\pi_{n}$ and $\mathfrak{p}$ is injective by definition,
so if $p(x_{n})=0$ for all $n$ then $\pi_{n}(x_{n})=x_{n-1}=0$
for all $n$, so $x=(x_{n})=0$. 

Exactness at the last term is the statement that $\pi$ is surjective,
which follows by construction, since $p=\mathfrak{p}\circ\pi$, $\mathfrak{p}$
is injective and the image of $\mathfrak{p}:W_{n}\Omega^{\cdot}\to W_{n+1}\Omega^{\cdot}$
coincides with the image of multiplication by $p$. 

Now we check that $\ker\pi=dV^{n}W\Omega^{\cdot}+V^{n}W\Omega^{\cdot}$.
Recall that $\pi_{n}:W_{n+1}\to W_{n}$ is the canonical projection.
It is enough to show that $\ker\pi_{n}=dV^{n}W_{1}\Omega^{\cdot}+V^{n}W_{1}\Omega^{\cdot}$.
First, if $x=V^{n}a+dV^{n}b\in W_{n+1}\Omega,$ it suffices to check
that $px=0$ and indeed $px=FV^{n+1}a+dFV^{n+1}b=0$. Now, let $[x]_{n+1}\in\ker\pi_{n}$,
where $x$ is an element of $\Omega_{\underline{Z}}^{\cdot}$ modulo
$p^{n+1}$. Then $[px]_{n+1}=p[x]_{n+1}=0,$ so it must be the case
that $px=p^{n+1}a+db$. We get $db=0$ mod $p$, so by the description
of cocycles mod $p$ we have $b=pb'+Fb''+db''$, so that $db=pdb'+pFdb''$.
Thus, \[
[x]_{n+1}=[p^{n}a]_{n+1}+[db']_{n+1}+[Fdb'']_{n+1}=\]
\[
=V^{n}[a]_{n+1}+d[p^{n}Fb'']_{n+1}=V^{n}[a]+dV^{n}[Fb''].\]

Now we check exactness at the second term. First, note that the sequence\[
W_{2n}\Omega^{q-1}\stackrel{F^{n}}{\to}W_{n}\Omega^{q-1}\stackrel{d}{\to}W_{n}\Omega^{q}\]
is exact, which is proved in the same way as Lemma 1.3.4 of \cite{Mo},
by taking the long exact sequence of cohomology sheaves of the short
exact sequence \[
0\to\Omega_{\underline{Z}}^{\cdot}/p^{n}\Omega_{\underline{Z}}^{\cdot}\stackrel{p^{n}}{\to}\Omega_{\underline{Z}}^{\cdot}/p^{2n}\Omega_{\underline{Z}}^{\cdot}\to\Omega_{\underline{Z}}^{\cdot}/p^{n}\Omega_{\underline{Z}}^{\cdot}\to0.\]
We note that the proof of the analogous statement in the classical
case in \cite{I} I (3.21) is wrong and corrected in \cite{IR} II
(1.3). Nakkajima proves this statement for formal de Rham-Witt complexes
in \cite{Na} 6.28 (6), using the same argument as Lemma 1.3.4 of
\cite{Mo}. 

We now claim that the projection \[
W\Omega^{\cdot}/p^{n}W\Omega^{\cdot}\to W_{n}\Omega^{\cdot}\]
is a quasi-isomorphism. This implies that \[
d^{-1}(p^{n}W\Omega^{q})=F^{n}W\Omega^{q-1}.\]
so if $dV^{n}x+V^{n}y=0$, then $dx+p^{n}y=0$, which in turn implies
$x=F^{n}z$ and $y=-F^{n}dz$ for some $z\in W\Omega^{q-1}$. This
checks exactness at the second term. Moreover, the fact that \[
W\Omega^{\cdot}/p^{n}W\Omega^{\cdot}\to W_{n}\Omega^{\cdot}\]
is a quasi-isomorphism follows in the same way as corollary 3.17 of
\cite{I}, boiling down to the Cartier isomorphism and to the description
of $\ker\pi$ as $dV^{n}+V^{n}$.\end{proof}
\begin{rem}
We note that one can use the Cartier isomorphisms to check properties
6.0.1 through 6.0.5 of \cite{Na} for $\Omega_{\underline{Z}}^{\cdot},\Omega_{\underline{Z}}^{\cdot}(-\log Y^{i})$
and $\Omega_{\underline{Z}}^{\cdot}(-\log Y^{1}-\log Y^{2})$, thus
proving the analogue of Proposition 6.27 of loc. cit. for all three
complexes. Then Theorem 6.24 of \cite{Na} also implies Lemma \ref{tensor with R_n}. 
\end{rem}

\subsection{The weight filtration}

The goal of this section is to define a double filtration $P_{k,l}$
on $W\tilde{\tilde{\omega}}_{Y}^{\cdot}$, which will be an analogue
of the weight filtration defined by Mokrane on $W_{n}\tilde{\omega}_{Y}^{\cdot}$
in the semistable case (see section 3 of \cite{Mo}). 

Let $(Z,\tilde{N})$ be an admissible lifting of $(Y,\tilde{M})$
over $(W[\tau,\sigma],\mathbb{N}^{2})$. We know that such liftings
exist etale locally. Let $\mathcal{Z}_{n}=Z\times_{W}W_{n}$. Let
$\tilde{N}_{1}$ be the log structure on $Z$ (or $\mathcal{Z}_{n}$)
obtained by pulling back the log structure on $Z_{r,s,m}$ associated
to \[
\mathbb{N}^{r}\to W[X_{1},\dots,X_{n},Y_{1},\dots,Y_{n},Z_{1},\dots Z_{m}]\]
\[
(0,\dots,0,1,0,\dots,0)\mapsto X_{i}\]
when $1$ is in the $i$th position. Define $\tilde{N}_{2}$ analogously.
The pullback of $\tilde{N}_{i}$ to $Y$ is the same as $\tilde{M}_{i}$.
For $i=1,2$, we have maps of sheaves of monoids $\tilde{N}_{i}\to\tilde{N}$. 

We define the following filtration on $\omega_{(\mathcal{Z}_{n},\tilde{N})/(W_{n},triv)}^{\cdot}$:
\[
P_{i,j}\omega_{(\mathcal{Z}_{n},\tilde{N})/(W_{n},triv)}^{q}:=\mbox{Im}(\omega_{(\mathcal{Z}_{n},\tilde{N}_{1})/(W_{n},triv)}^{i}\otimes\omega_{(\mathcal{Z}_{n},\tilde{N_{2}})/(W_{n},triv)}^{j}\otimes\Omega_{\mathcal{Z}_{n}/k}^{q-i-j}\to\omega_{(\mathcal{Z}_{n},\tilde{N})/(W_{n},triv)}^{q})\]
for $i,j\geq0$ and $i+j\leq q$. This filtration respects the differential
and induces a filtration $P_{i,j}\tilde{\tilde{C}}_{Y}^{\cdot}$ on
$\tilde{\tilde{C}}_{Y}^{\cdot}$ (which can be thought of as a quotient
of $\omega_{(\mathcal{Z}_{n},\tilde{M})/(W_{n},triv)}^{\cdot}$, as
in the proof of Lemma \ref{lem:Cartier}). Note that if we let \[
P_{k}\omega_{(\mathcal{Z}_{n},\tilde{N})/(W_{n},triv)}^{q}=\mbox{Im}(\omega_{(\mathcal{Z}_{n},\tilde{N})/(W_{n},triv)}^{k}\otimes\Omega_{\mathcal{Z}_{n}/k}^{q-k}\to\omega_{(\mathcal{Z}_{n},\tilde{N})/(W_{n},triv)}^{q})\]
then $ $$P_{k}$ is the weight filtration defined in 1.1.1 of \cite{Mo}
and $P_{i,j}\omega_{(\mathcal{Z}_{n},\tilde{N})/(W_{n},triv)}^{\cdot}\subset P_{i+j}\omega_{(\mathcal{Z}_{n},\tilde{N})/(W_{n},triv)}^{\cdot}$. 

For $i=1,\dots,r$, let $D_{1,i}$ be the pullback to $Z$ of the
divisor of $Z_{r,s,m}$ obtained by setting $X_{i}=0$. Similarly,
for $i=1,\dots,s$, let $D_{2,i}$ be the pullback to $Z$ of the
divisor of $Z_{r,s,m}$ obtained by setting $Y_{i}=0.$ For $i,j\geq0$
let $D^{(i,j)}$ be the disjoint union of $ $\[
D_{1,k_{1}}\times_{Z}\dots\times_{Z}D_{1,k_{i}}\times_{Z}D_{2,l_{1}}\times_{Z}\dots\times_{Z}D_{2,l_{j}},\]
over all $k_{1},\dots,k_{i}\in\{1,\dots,r\}$ and $l_{1},\dots,l_{j}\in\{1,\dots,s\}$.
And let $\tau_{i,j}:D^{(i,j)}\to Z$ be the obvious morphism, with
$\mathcal{D}_{n}^{(i,j)},\tau_{i,j}$ the pullbacks to $\mathcal{Z}_{n}$
. Let \[
\mbox{Gr}_{i,j}\omega_{(\mathcal{Z}_{n},\tilde{N})/(W_{n},triv)}^{q}:=P_{i,j}\omega_{(\mathcal{Z}_{n},\tilde{N})/(W_{n},triv)}^{q}/(P_{i-1,j}\omega_{(\mathcal{Z}_{n},\tilde{N})/(W_{n},triv)}^{q}+P_{i,j-1}\omega_{(\mathcal{Z}_{n},\tilde{N})/(W_{n},triv)}^{q}).\]

For $i,j\geq1$ we will define a morphism of sheaves\[
\mbox{Res}:\mbox{Gr}_{i,j}\omega_{(\mathcal{Z}_{n},\tilde{N})/(W_{n},triv)}^{q}\to(\tau_{i,j})_{*}\Omega_{\mathcal{D}_{n}^{(i,j)}/W_{n}}^{q-i-j},\]
which extends to a morphism of complexes. If $\omega=\alpha\wedge\frac{dX_{k_{1}}}{X_{k_{1}}}\wedge\dots\wedge\frac{dX_{k_{i}}}{X_{k_{i}}}\wedge\frac{dY_{l_{1}}}{Y_{l_{1}}}\wedge\dots\wedge\frac{dY_{k_{j}}}{Y_{k_{j}}}$
is a local section of $P_{i,j}\omega_{(\mathcal{Z}_{n},\tilde{N})/(W_{n},triv)}^{q}$
with $k_{1}<\dots<k_{i}$ and $l_{1}<\dots<l_{j}$, then \[
\mbox{Res}(\omega):=\alpha|_{D_{1,k_{1}}\times_{Z}\dots\times_{Z}D_{1,k_{i}}\times_{Z}D_{2,l_{1}}\times_{Z}\dots\times_{Z}D_{2,l_{j}}}.\]
This factors through $P_{i-1,j}+P_{i,j-1}$ and extends to a global
map of sheaves. 

Alternatively, we can follow the construction in section 3 of chapter
II of \cite{D}. Let $\mathcal{D}_{n}^{k}$ be the disjoint union
of intersections of $k$ divisors $D_{j,k_{i}}$ with $j=1,2$ and
$k_{i}\in\{1,\dots,n\}$. These intersections are in one-to-one correspondence
with images of injections \[
f:\{1,\dots,k\}\to\{1,\dots,n\}\cup\{1,\dots,n\}\]
and so we denote one of these $k$ intersections by $\mathcal{D}_{n}^{f}$
(even though it only really depends on $\mbox{Im}f$). We have $\mathcal{D}_{n}^{k}=\bigsqcup_{\substack{i+j=k\\
i,j\geq0}
}\mathcal{D}_{n}^{i,j}=\bigsqcup_{\mbox{Im}f}\mathcal{D}_{n}^{f}.$ Let $\tau_{f}:\mathcal{D}_{n}^{f}\to\mathcal{Z}_{n}$ be the closed
immersion. In 3.5.2 of \cite{D}, a morphism \[
\rho_{1}:(\tau_{f})_{*}\Omega_{\mathcal{D}_{n}^{f}}^{q-k}\to P_{k}\omega_{(\mathcal{Z}_{n},\tilde{N})/(W_{n},triv)}^{q}/P_{k-1}\]
(and then a morphism $\rho_{2}$, which dependes on an ordering of
$\{1,\dots,n\}\cup\{1,\dots,n\}$) is associated to each such injection
and the sum of $\rho_{2}$ over all injections $f$ determines an
isomorphism\[
\rho:(\tau_{k})_{*}\Omega_{\mathcal{D}_{n}^{k}/W_{n}}^{\cdot}[-k]\stackrel{\sim}{\to}P_{k}\omega_{(\mathcal{Z}_{n},\tilde{N})/(W_{n},triv)}^{q}/P_{k-1}\]
 by Proposition 3.6 of Chapter II of \cite{D}.

We are only interested in injections $q_{i,j}:\{1,\dots,i+j\}\to\{1,\dots,n\}\cup\{1,\dots,n\}$
with image of cardinality $i$ in the first $\{1,\dots,n\}$ term
and cardinality $j$ in the second $\{1,\dots,n\}$ term. We let $\mbox{Res}^{-1}$
be the sum of the morphisms $\rho_{2}$ over all injections $q_{i,j}$.
When we have an injection of type $q_{i,j}$, the image of the morphism
$\rho_{2}$ defined by Deligne falls in $ $\[
P_{i,j}\omega_{(\mathcal{Z}_{n},\tilde{N})/(W_{n},triv)}^{q}/(P_{i-1,j}+P_{i,j-1})\subset P_{i+j}\omega_{(\mathcal{Z}_{n},\tilde{N})/(W_{n},triv)}^{q}/P_{i+j-1}.\]
For $k\geq1$, we have the direct sum decompositions \[
P_{k}\omega_{(\mathcal{Z}_{n},\tilde{N})/(W_{n},triv)}^{\cdot}/P_{k-1}=\bigoplus_{\substack{i+j=k\\
i,j\geq0}
}\mbox{Gr}_{i,j}\omega_{(\mathcal{Z}_{n},\tilde{N})/(W_{n},triv)}^{\cdot}\mbox{ and }\]
\[
(\tau_{k})_{*}\Omega_{\mathcal{D}_{n}^{(k)}/W_{n}}^{q-k}=\bigoplus_{\substack{i+j=k\\
i,j\geq0}
}(\tau_{i,j})_{*}\Omega_{\mathcal{D}_{n}^{(i,j)}/W_{n}}^{q-i-j}.\]
It is easy to check that the isomorphism $\rho$ matches up the $(i,j)$
terms in each decomposition. Putting this discussion together, we
get the following. 
\begin{lem}
For $i,j\geq1$, the map\[
\mbox{Res}^{-1}:(\tau_{i,j})_{*}\Omega_{\mathcal{D}_{n}^{(i,j)}/W_{n}}^{q-i-j}\to\mbox{Gr}_{i,j}\omega_{(\mathcal{Z}_{n},\tilde{N})/(W_{n},triv)}^{q}\]
is an isomorphism. 
\end{lem}
We also have the following analogue of Lemma 1.2 of \cite{Mo}.
\begin{lem}
We have an exact sequence of complexes \[
0\to P_{i-1,j-1}\omega_{(\mathcal{Z}_{n},\tilde{N})/(W_{n},triv)}^{\cdot}\to P_{i-1,j}\omega_{(\mathcal{Z}_{n},\tilde{N})/(W_{n},triv)}^{\cdot}\oplus P_{i,j-1}\omega_{(\mathcal{Z}_{n},\tilde{N})/(W_{n},triv)}^{\cdot}\to\]
\[
\to P_{i,j}\omega_{(\mathcal{Z}_{n},\tilde{N})/(W_{n},triv)}^{\cdot}\to\mbox{Gr}_{i,j}\omega_{(\mathcal{Z}_{n},\tilde{N})/(W_{n},triv)}^{\cdot}\to0.\]
The long exact cohomology sequence(s) associated to this have all
coboundaries $0$, so we get the exact sequence:\[
0\to\mathcal{H}^{q}(P_{i-1,j-1}\omega_{(\mathcal{Z}_{n},\tilde{N})/(W_{n},triv)}^{\cdot})\to\mathcal{H}^{q}(P_{i-1,j}\omega_{(\mathcal{Z}_{n},\tilde{N})/(W_{n},triv)}^{\cdot})\oplus\mathcal{H}^{q}(P_{i,j-1}\omega_{(\mathcal{Z}_{n},\tilde{N})/(W_{n},triv)}^{\cdot})\to\]
\[
\to\mathcal{H}^{q}(P_{i,j}\omega_{(\mathcal{Z}_{n},\tilde{N})/(W_{n},triv)}^{\cdot})\to\mathcal{H}^{q}(\Omega_{\mathcal{D}_{n}^{(i,j)}/W_{n}}^{\cdot}[-i-j])\to0.\]
\end{lem}
\begin{proof}
The first assertion is clear. In order to show that the second sequence
is exact, it suffices to show the following two statements about cocycles:
\begin{enumerate}
\item $ZP_{i,j}\omega_{(\mathcal{Z}_{n},\tilde{N})/(W_{n},triv)}^{q}\twoheadrightarrow Z\Omega_{\mathcal{D}_{n}^{(i,j)}/W_{n}}^{q-i-j}.$
\item $ZP_{i-1,j}\omega_{(\mathcal{Z}_{n},\tilde{N})/(W_{n},triv)}^{q}\oplus ZP_{j,i-1}\omega_{(\mathcal{Z}_{n},\tilde{N})/(W_{n},triv)}^{q}\twoheadrightarrow Z(P_{i-1,j}\omega_{(\mathcal{Z}_{n},\tilde{N})/(W_{n},triv)}^{q}+P_{i,j-1}\omega_{(\mathcal{Z}_{n},\tilde{N})/(W_{n},triv)}^{q}).$
\end{enumerate}
The first statement is proved in the same way as the main step in
Lemma 1.1.2 of \cite{Mo}. If $\alpha$ is a local section of $Z\Omega_{\mathcal{D}_{n}^{(i,j)}/W_{n}}^{q-i-j},$
assume that $\alpha$ is supported on some \[
D_{1,k_{1}}\times_{Z}\dots\times_{Z}D_{1,k_{i}}\times_{Z}D_{2,l_{1}}\times_{Z}\dots\times_{Z}D_{2,l_{j}},\]
for some $k_{1},\dots,k_{i},l_{1},\dots,l_{j}\in\{1,\dots,n\}$. Let
\[
\rho:\mathcal{Z}_{n}\to D_{1,k_{1}}\times_{Z}\dots\times_{Z}D_{1,k_{i}}\times_{Z}D_{2,l_{1}}\times_{Z}\dots\times_{Z}D_{2,l_{j}}\]
be the retraction associated to the immersion \[
D_{1,k_{1}}\times_{Z}\dots\times_{Z}D_{1,k_{i}}\times_{Z}D_{2,l_{1}}\times_{Z}\dots\times_{Z}D_{2,l_{j}}\to\mathcal{Z}_{n}.\]
Then $\rho^{*}\alpha$ lifts $\alpha$ to a section of $Z\Omega_{\mathcal{Z}_{n}/W_{n}}^{q-i-j}$
and the section $\omega_{\alpha}=\rho^{*}\alpha\wedge\frac{dX_{k_{1}}}{X_{k_{1}}}\wedge\dots\wedge\frac{dX_{k_{i}}}{X_{k_{i}}}\wedge\frac{dY_{l_{1}}}{Y_{l_{1}}}\wedge\dots\wedge\frac{dY_{l_{j}}}{Y_{l_{j}}}\in P_{i,j}\omega_{(\mathcal{Z}_{n},\tilde{N})/(W_{n},triv)}^{q}$
satisfies $d\omega=0$ and $\mbox{Res}(\omega)=\alpha$. From this,
we know that the coboundaries of the long exact sequence associated
to \[
0\to P_{i-1,j}\omega_{(\mathcal{Z}_{n},\tilde{N})/(W_{n},triv)}^{\cdot}+P_{i,j-1}\omega_{(\mathcal{Z}_{n},\tilde{N})/(W_{n},triv)}^{\cdot}\to P_{i,j}\omega_{(\mathcal{Z}_{n},\tilde{N})/(W_{n},triv)}^{\cdot}\to\mbox{Gr}_{i,j}\omega_{(\mathcal{Z}_{n},\tilde{N})/(W_{n},triv)}^{\cdot}\to0\]
are $0$, so we also know that \[
\mathcal{H}^{q}(P_{i-1,j}\omega_{(\mathcal{Z}_{n},\tilde{N})/(W_{n},triv)}^{\cdot}+P_{i,j-1}\omega_{(\mathcal{Z}_{n},\tilde{N})/(W_{n},triv)}^{\cdot})\hookrightarrow\mathcal{H}^{q}(P_{i,j}\omega_{(\mathcal{Z}_{n},\tilde{N})/(W_{n},triv)}^{\cdot})\]
 for every $i,j\geq1$. 

For the second statement, we have to prove that if $\alpha\in P_{i-1,j}\omega_{(\mathcal{Z}_{n},\tilde{N})/(W_{n},triv)}^{q}$
and $\beta\in P_{i,j-1}\omega_{(\mathcal{Z}_{n},\tilde{N})/(W_{n},triv)}^{q}$
satisfy $d(\alpha+\beta)=0$ then we can find $\alpha'\in ZP_{i-1,j}\omega_{(\mathcal{Z}_{n},\tilde{N})/(W_{n},triv)}^{q}$
and $\beta'\in ZP_{i,j-1}\omega_{(\mathcal{Z}_{n},\tilde{N})/(W_{n},triv)}^{q}$
such that $\alpha'+\beta'=\alpha+\beta$. If $\alpha\in P_{i-1,j-1}\omega_{(\mathcal{Z}_{n},\tilde{N})/(W_{n},triv)}^{q}$
then we are done, since we can just take $\alpha'=0,\beta'=\alpha+\beta$.
The same holds for $\beta$. Otherwise, we have $d\alpha\in P_{i-1,j-1}$
so by the injectivity proved in statement 1 for $(i-1,j)$, we know
that $d\alpha=d\alpha_{1}+d\alpha_{2}$ for some $\alpha_{1}\in P_{i-1,j-1}$
and $\alpha_{2}\in P_{i-2,j}$. Thus, we've reduced our problem from
$(i-1,j)$ to $(i-2,j)$. Proceeding by induction, we may assume that
$i=0$. In that case $d\alpha_{2i}\in P_{0,j-1}$. By (the same argument
as in the proof of) Lemma 1.1.2 of \cite{Mo}, we have an injection
\[
\mathcal{H}^{q}(P_{0,j-1}\omega_{(\mathcal{Z}_{n},\tilde{N})/(W_{n},triv)}^{\cdot})\hookrightarrow\mathcal{H}^{q}(P_{0,j}\omega_{(\mathcal{Z}_{n},\tilde{N})/(W_{n},triv)}^{\cdot}),\]
so that implies $d\alpha_{2i}=d\alpha_{2i+1}$ for some $\alpha_{2i+1}\in P_{0,j-1}$.
Then \[
\alpha':=\alpha-(\sum_{i'=0}^{i}\alpha_{2i'+1})\in ZP_{i-1,j},\beta':=\beta+\sum_{i'=0}^{i}\alpha_{2i'+1}\in ZP_{i,j-1}\]
satisfy the desired relations. 
\end{proof}
The double filtration $P_{i,j}$ on $\omega_{(\mathcal{Z}_{n},\tilde{N})/(W_{n},triv)}^{\cdot}$
induces a double filtration $P_{i,j}$ on $\tilde{\tilde{C}}_{\mathcal{Z}_{n}}$
and for $i,j\geq1$ the residue morphism $\mathrm{Res}:P_{i,j}\omega_{(\mathcal{Z}_{n},\tilde{N})/(W_{n},triv)}^{q}\to\Omega_{\mathcal{D}_{n}^{(i,j)}/W_{n}}^{q-i-j}$
factors through $P_{i,j}\tilde{\tilde{C}}_{\mathcal{Z}_{n}}$.
\begin{lem}
For any two admissible liftings $(Z_{1},\tilde{N})$ and $(Z_{2},\tilde{N})$
of $(Y,\tilde{M})$ we have a canonical isomorphism\[
\alpha_{Z_{1}Z_{2}}:\mathcal{H}^{q}(P_{i,j}\tilde{\tilde{C}}_{\mathcal{Z}_{1,n}}^{\cdot})\to\mathcal{H}^{q}(P_{i,j}\tilde{\tilde{C}}_{\mathcal{Z}_{2,n}}^{\cdot})\]
satisfying the cocycle condition for any three admissible liftings.

Moreover, the residue morphism $\mbox{Res}_{Z}:\mathcal{H}^{q}(P_{i,j}\tilde{\tilde{C}}_{\mathcal{Z}_{n}}^{\cdot})\to\mathcal{H}^{q-i-j}(\Omega_{\mathcal{D}_{n}^{(i,j)}/W_{n}}^{\cdot})\simeq W_{n}\Omega_{Y^{(i,j)}}^{q}$
induced on cohomology satisfies the compatibility \[
\mbox{Res}_{Z_{1}}=\mbox{Res}_{Z_{2}}\circ\alpha_{Z_{1}Z_{2}}.\]
\end{lem}
\begin{proof}
The proof of the first part is basically the same as the proof of
Lemma \ref{lem:independence}. We take admissible lifts $(Z_{1},\tilde{N})$
and $(Z_{2},\tilde{N})$ (we denote the log structures on both simply
by $\tilde{N}$, as it will be understood from the context which is
the underlying scheme). As in the proof of Lemma \ref{lem:independence},
we form $(Z_{12},\tilde{N})$, which is smooth over $(Z_{i},\tilde{N})$,
even though it is not quite an admissible lift. However, $Z_{12}$
is etale over \[
\mbox{Spec }W[X_{1},\dots,X_{n},Y_{1},\dots,Y_{n},X'_{1},\dots X'_{n},Y'_{1},\dots,Y'_{n},v_{1}^{\pm1},\dots,v_{r}^{\pm1},u_{1}^{\pm1},\dots u_{s}^{\pm1}]/(X_{i}v_{i}-X'_{i}.Y_{j}v_{j}-Y'_{j}).\]
So we can endow $\tilde{\tilde{C}}_{\mathcal{Z}_{12,n}}^{\cdot}$
with a filtration $P_{i,j}\tilde{\tilde{C}}_{\mathcal{Z}_{12},n}^{\cdot}$
defined as above, in terms of log structures $\tilde{N}_{1}$ and
$\tilde{N}_{2}$ (which come from formally {}``inverting'' the $X_{i}$
and $X'_{i}$ or the $Y_{i}$ and $Y_{i}'$). Then the same argument
used in the proof of Lemma \ref{lem:independence} gives us quasi-isomorphisms\[
P_{i,j}\tilde{\tilde{C}}_{\mathcal{Z}_{i,n}}\to P_{i,j}\tilde{\tilde{C}}_{\mathcal{Z}_{12,n}}\]
for $i=1,2$, which satisfy the right compatibility condition for
three admissible lifts.

For the second part, we follow the argument in Lemma 3.4 (2) of \cite{Mo}.
We let \[
\omega=\alpha\wedge\frac{dX_{k_{1}}}{X_{k_{1}}}\wedge\dots\wedge\frac{dX_{k_{i}}}{X_{k_{i}}}\wedge\frac{dY_{l_{1}}}{Y_{l_{1}}}\wedge\dots\wedge\frac{dY_{l_{i}}}{Y_{l_{i}}}\]
be a section of $P_{i,j}\omega_{(\mathcal{Z}_{1,n},\tilde{N})/(W_{n},triv)}^{q}$
and \[
\omega'=\alpha'\wedge\frac{dX'_{k_{1}}}{X'_{k_{1}}}\wedge\dots\wedge\frac{dX'_{k_{i}}}{X'_{k_{i}}}\wedge\frac{dY'_{l_{1}}}{Y'_{l_{1}}}\wedge\dots\wedge\frac{dY'_{l_{i}}}{Y'_{l_{i}}}\]
be a section of $P_{i,j}\omega_{(\mathcal{Z}_{2,n},\tilde{N})/(W_{n},triv)}^{q}$
such that $\omega=\omega'$ in $P_{i,j}\omega_{(\mathcal{Z}_{12,n},\tilde{N})/(W_{n},triv)}^{q}$.
We have to check that $\alpha|_{\mathcal{D}_{12,n}^{(i,j)}}=\alpha'|_{\mathcal{D}_{12,n}^{(i,j)}}$.
But \[
\omega-\omega'=(\alpha-\alpha')\wedge\frac{dX_{k_{1}}}{X_{k_{1}}}\wedge\dots\wedge\frac{dX_{k_{i}}}{X_{k_{i}}}\wedge\frac{dY_{l_{1}}}{Y_{l_{1}}}\wedge\dots\wedge\frac{dY_{l_{i}}}{Y_{l_{i}}}+\Psi,\]
where $\Psi\in P_{i,j-1}\omega_{(\mathcal{Z}_{2,n},\tilde{N})/(W_{n},triv)}^{q}+P_{i-1,j}\omega_{(\mathcal{Z}_{2,n},\tilde{N})/(W_{n},triv)}^{q}$.
This means that \[
(\alpha-\alpha')\wedge\frac{dX_{k_{1}}}{X_{k_{1}}}\wedge\dots\wedge\frac{dX_{k_{i}}}{X_{k_{i}}}\wedge\frac{dY_{l_{1}}}{Y_{l_{1}}}\wedge\dots\wedge\frac{dY_{l_{i}}}{Y_{l_{i}}}\]
is also a section of $P_{i,j-1}\omega_{(\mathcal{Z}_{2,n},\tilde{N})/(W_{n},triv)}^{q}+P_{i-1,j}\omega_{(\mathcal{Z}_{2,n},\tilde{N})/(W_{n},triv)}^{q}$,
so $(\alpha-\alpha')|_{\mathcal{D}_{12,n}^{(i,j)}}=0$. \end{proof}
\begin{cor}
\label{identifying graded pieces}We can define the sheaves \[
P_{i,j}W_{n}\tilde{\tilde{\omega}}_{Y}^{q}:=\mathcal{H}^{q}(P_{i,j}\tilde{\tilde{C}}_{Y}^{\cdot}).\]
The complexes $P_{i,j}W_{n}\tilde{\tilde{\omega}}_{Y}^{\cdot}$ form
an increasing double filtration of $W_{n}\tilde{\tilde{\omega}}_{Y}^{\cdot}$
such that the graded pieces for $i,j\geq1$ \[
\mbox{Gr}_{i,j}W_{n}\tilde{\tilde{\omega}}_{Y}^{\cdot}:=P_{i,j}W_{n}\tilde{\tilde{\omega}}_{Y}^{\cdot}/P_{i,j-1}+P_{i-1,j}\]
are canonically isomorphic to the de Rham Witt complexes of the smooth
subschemes $Y^{(i,j)}$:\[
\mbox{Res}:\mbox{Gr}_{i,j}W_{n}\tilde{\tilde{\omega}}_{Y}^{\cdot}\stackrel{\sim}{\to}W_{n}\Omega_{Y^{(i,j)}}^{\cdot}[-i-j](-i-j).\]
\end{cor}
\begin{lem}
\label{compatibility with projections}The constructions in this sections
are compatible with the transition morphisms $\pi$, in the following
way. 
\begin{enumerate}
\item The following diagrams are commutative: \[
\xymatrix{W_{n+1}\tilde{\tilde{\omega}}_{Y}^{q}\ar[d]_{\wedge\frac{d\tau}{\tau}}\ar[r]^{\pi} & W_{n}\tilde{\tilde{\omega}}_{Y}^{q}\ar[d]^{\wedge\frac{d\tau}{\tau}}\\
W_{n+1}\tilde{\tilde{\omega}}_{Y}^{q}\ar[r]^{\pi} & W_{n}\tilde{\tilde{\omega}}_{Y}^{q}}
\]
and \[
\xymatrix{W_{n+1}\tilde{\tilde{\omega}}_{Y}^{q}\ar[d]_{\wedge\frac{d\sigma}{\sigma}}\ar[r]^{\pi} & W_{n}\tilde{\tilde{\omega}}_{Y}^{q}\ar[d]^{\wedge\frac{d\sigma}{\sigma}}\\
W_{n+1}\tilde{\tilde{\omega}}_{Y}^{q}\ar[r]^{\pi} & W_{n}\tilde{\tilde{\omega}}_{Y}^{q}}
.\]

\item The projection $\pi:W_{n+1}\tilde{\tilde{\omega}}_{Y}^{q}\to W_{n}\tilde{\tilde{\omega}}_{Y}^{q}$
preserves the weight filtration $P_{i,j}$ on $W_{m}\tilde{\tilde{\omega}}_{Y}^{q}$
for $m=n,n+1$.
\item The morphism $\pi:P_{i,j}W_{n+1}\tilde{\tilde{\omega}}_{Y}^{q}\to P_{i,j}W_{n}\tilde{\tilde{\omega}}_{Y}^{q}$
is surjective. 
\end{enumerate}
\end{lem}
\begin{proof}
The first part follows in the same way as Proposition 8.1 of \cite{Na},
by using a local admissible lifting $(Z,\tilde{N})$ of $(Y,\tilde{M})$
together with a lift of Frobenius $\Phi$. Then $\Phi^{*}(\tau)=\tau^{p}(1+pu)$
for some $u\in\mathcal{O}_{Z}\otimes_{W[\tau,\sigma]}W_{n}<\tau,\sigma>$
and so $\Phi^{*}(d\log\tau)$ is equivalent to $pd\log\tau$ modulo
an exact form. The same holds for $\sigma$. 

The second part follows in the same way as Proposition 8.4 of \cite{Na}.
The question is local, so we may assume that the admissible lift $(Z,\tilde{N})$
is etale over $\mathrm{Spec\ }W[X_{1},\dots,X_{n},Y_{1},\dots,Y_{n}],\mathbb{N}^{r}\oplus\mathbb{N}^{s}.$
First we see that, for a lift $\Phi$ of Frobenius we have that $\Phi^{*}(d\log X_{i})$
is equivalent modulo an exact form to $pd\log X_{i}$ for $1\leq i\leq r$
and that $\Phi^{*}(d\log Y_{j})$ is equivalent modulo an exact form
to $pd\log Y_{j}$ for $1\leq j\leq s$. This implies that the map
$\mathfrak{p}:W_{n}\tilde{\tilde{\omega}}_{Y}^{q}\to W_{n+1}\tilde{\tilde{\omega}}_{Y}^{q}$
preserves the weight filtration $P_{i,j}$. 

In order to see that $\pi:W_{n+1}\tilde{\tilde{\omega}}_{Y}^{q}\to W_{n}\tilde{\tilde{\omega}}_{Y}^{q}$
also preserves $P_{i,j}$ we use a descending induction on $(i,j)$
in lexicographic order. Note that $P_{r,s}W_{n}\tilde{\tilde{\omega}}_{Y}^{q}=W_{n}\tilde{\tilde{\omega}}_{Y}^{q}$,
so there is nothing to prove in this case. We can prove the result
for $(r,s-1)$ in the same way as Proposition 8.4 (2) of \cite{Na},
using the commutative diagrams \[
\xymatrix{P_{i,j}W_{n+1}\tilde{\tilde{\omega}}_{Y}^{q}\ar[d]_{\pi}\ar[r]^{\mathrm{Res}} & W_{n+1}\Omega_{Y^{(i,j)}}^{q-i-j}\ar[d]^{\pi}\\
P_{i,j}W_{n}\tilde{\tilde{\omega}}_{Y}^{q}\ar[r]^{\mathrm{Res}} & W_{n}\Omega_{Y^{(i,j)}}^{q-i-j}}
\]
for $(i,j)$ successively equal to $(r,s),(r-1,s),\dots,(1,s)$. At
the last step we get a commutative diagram of exact sequences \[
\xymatrix{0\ar[r] & P_{r,s-1}W_{n+1}\tilde{\tilde{\omega}}_{Y}^{q}+P_{0,s}W_{n+1}\tilde{\tilde{\omega}}_{Y}^{q}\ar[r] & P_{r,s-1}W_{n+1}\tilde{\tilde{\omega}}_{Y}^{q}+P_{1,s}W_{n+1}\tilde{\tilde{\omega}}_{Y}^{q}\ar[d]_{\pi}\ar[r] & W_{n+1}\Omega_{Y^{(1,s)}}^{q-s-1}\ar[r]\ar[d]_{\pi} & 0\\
0\ar[r] & P_{r,s-1}W_{n}\tilde{\tilde{\omega}}_{Y}^{q}+P_{0,s}W_{n}\tilde{\tilde{\omega}}_{Y}^{q}\ar[r] & P_{r,s-1}W_{n}\tilde{\tilde{\omega}}_{Y}^{q}+P_{1,s}W_{n}\tilde{\tilde{\omega}}_{Y}^{q}\ar[r] & W_{n}\Omega_{Y^{(1,s)}}^{q-s-1}\ar[r] & 0}
,\]
which means there is an induced morphism $\pi:P_{r,s-1}W_{n+1}\tilde{\tilde{\omega}}_{Y}^{q}+P_{0,s}W_{n+1}\tilde{\tilde{\omega}}_{Y}^{q}\to P_{r,s-1}W_{n}\tilde{\tilde{\omega}}_{Y}^{q}+P_{0,s}W_{n}\tilde{\tilde{\omega}}_{Y}^{q}.$

At this stage, we note that we can define \[
Y^{(0,s)}=\bigsqcup_{\substack{T\subseteq\{1,\dots,n\}\\
\#T=s}
}\left(\bigcap_{i\in T}Y_{i}^{2}\right).\]
This will be a simple reduced normal crossings divisor over $k$ and
we can endow it with the pullback of the log structure $\tilde{M}_{1}$
so that $(Y,\tilde{M})$ is a $(k,\mathbb{N})$-semistable log scheme,
in the terminology of section 2.4 of \cite{Mo}. There is a surjective
residue mophism obtained via restriction \[
P_{i,j}W_{n}\tilde{\tilde{\omega}}_{Y}^{q}\stackrel{\mathrm{Res}}{\to}P_{i}W_{n}\tilde{\omega}_{Y^{(0,j)}}^{q-j},\]
which respects the weight filtrations. Just as the commutative diagram
8.4.3 of \cite{Na} is obtained, we can use the injectivity of $\mathfrak{p}:W_{n}\tilde{\omega}_{Y^{(0,s)}}^{q}\to W_{n+1}\tilde{\omega}_{Y^{(0,s)}}^{q}$
for $Y^{(0,s)}/k$ (Corollary 6.28 (2) of \cite{Na}) to see that
there is a commutative diagram\[
\xymatrix{P_{0,s}W_{n+1}\tilde{\tilde{\omega}}_{Y}^{q}\ar[d]_{\pi}\ar[r]^{\mathrm{Res}} & P_{0}W_{n+1}\tilde{\omega}_{Y^{(0,s)}}^{q-s}\ar[d]^{\pi}\\
P_{0,s}W_{n}\tilde{\tilde{\omega}}_{Y}^{q}\ar[r]^{\mathrm{Res}} & P_{0}W_{n}\tilde{\omega}_{Y^{(0,s)}}^{q-s}}
.\]
 We therefore get a commutative diagram of exact sequences: \[
\xymatrix{0\ar[r] & P_{r,s-1}W_{n+1}\tilde{\tilde{\omega}}_{Y}^{q}\ar[r] & P_{r,s-1}W_{n+1}\tilde{\tilde{\omega}}_{Y}^{q}+P_{0,s}W_{n+1}\tilde{\tilde{\omega}}_{Y}^{q}\ar[d]_{\pi}\ar[r] & P_{0}W_{n+1}\tilde{\omega}_{Y^{(0,s)}}^{q-s}\ar[r]\ar[d]_{\pi} & 0\\
0\ar[r] & P_{r,s-1}W_{n}\tilde{\tilde{\omega}}_{Y}^{q}\ar[r] & P_{r,s-1}W_{n}\tilde{\tilde{\omega}}_{Y}^{q}+P_{0,s}W_{n}\tilde{\tilde{\omega}}_{Y}^{q}\ar[r] & P_{0}W_{n}\tilde{\omega}_{Y^{(0,s)}}^{q-s}\ar[r] & 0}
,\]
so there is an induced morphism $\pi:P_{r,s-1}W_{n+1}\tilde{\tilde{\omega}}_{Y}^{q}\to P_{r,s-1}W_{n}\tilde{\tilde{\omega}}_{Y}^{q}$.

Finally, the third part follows in the same way as Corollary 8.6.4
of \cite{Na}. For an admissible lift $(Z,\tilde{N}),$ let $Z_{1}:=Z\times_{W}k$.
We have surjective morphisms $W_{n}\Omega_{Z_{1}}^{q}\to P_{0,0}W_{n}\tilde{\tilde{\omega}}_{Y}^{q}$,
which commute with the transition morphisms $\pi$. So $\pi$ is surjective
for $P_{0,0}$. Using the exact sequences of the form\[
0\to P_{0,j-1}W_{n}\tilde{\tilde{\omega}}_{Y}^{q}\to P_{0,j}W_{n}\tilde{\tilde{\omega}}_{Y}^{q}\to P_{0}W_{n}\tilde{\omega}_{Y^{(0,j)}}^{q-j}\to0\]
and the surjectivity of $\pi$ on the third term, we prove by induction
on $j$ that $\pi$ is surjective for $P_{0,j}$. The same statement
holds for $P_{i,0}.$Then, we prove that $\pi$ is surjective for
a general $P_{i,j}$ by induction on $i+j$, using the exact sequences
of the form\[
0\to P_{i-1,j}W_{n}\tilde{\tilde{\omega}}_{Y}^{q}+P_{i,j-1}W_{n}\tilde{\tilde{\omega}}_{Y}^{q}\to P_{i,j}W_{n}\tilde{\tilde{\omega}}_{Y}^{q}\to W_{n}\Omega_{Y^{(i,j)}}^{q-i-j}\to0.\]

\end{proof}

\section{Generalizing the Mokrane spectral sequence}

We define a double complex $W_{n}A^{\cdot\cdot}$ as follows. Its
terms are\[
W_{n}A^{ij}:=\bigoplus_{k=0}^{j}W_{n}\tilde{\tilde{\omega}}_{Y}^{i+j+2}/P_{k,i+j+2}+P_{i+j+2,j-k}\mbox{ for }i,j\geq0\]
and $W_{n}A^{ij}:=0$ otherwise. The operators $d,\pi,F,V$ of $W_{\cdot}\tilde{\tilde{\omega}}^{\cdot}$
induce operators $d',\pi,F,V$ of the pro-complexes $W_{\cdot}A^{\cdot j}$.
For $x$ in the direct summand $W_{n}\tilde{\tilde{\omega}}_{Y}^{i+j+2}/P_{k,i+j+2}+P_{i+j+2,j-k}$
of $W_{n}A^{ij}$, $d'x$ is the class of $(-1)^{j}d\tilde{x}$, where
$\tilde{x}$ is a lift of $x$ in $W_{n}\tilde{\tilde{\omega}}_{Y}^{i+j+2}$.
We also have a differential $d'':W_{n}A^{ij}\to W_{n}A^{ij+1}$ given
by\[
d''x=(-1)^{i}\left(\frac{d\tau}{\tau}\wedge x+\frac{d\sigma}{\sigma}\wedge x\right),\]
where $\frac{d\tau}{\tau}$ and $\frac{d\sigma}{\sigma}$ are the
global sections of $W_{n}\tilde{\tilde{\omega}}_{Y}^{1}$ defined
in Lemma \ref{global sections}. We have $d'd''=d''d'$, so we indeed
get a double pro-complex $(W_{\cdot}A^{\cdot\cdot},d',d'').$ As in
Lemma 3.9 of \cite{Mo}, we can use devissage by weights to see that
the components of this pro-complex are $p$-torsion-free. Let $W_{\cdot}A^{\cdot}$
be the simple pro-complex associated to the double pro-complex $W_{\cdot}A^{\cdot\cdot}$. 

We define now an endomorphism $\nu$ of bidegree $(-1,1)$ of $W_{n}A^{\cdot\cdot}$
which will induce the monodromy operator on cohomology. For each $k\in\{0,\dots,j\}$
we have natural maps \[
W_{n}\tilde{\tilde{\omega}}_{Y}^{i+j+2}/P_{k,i+j+2}+P_{i+j+2,j-k}\to W_{n}\tilde{\tilde{\omega}}_{Y}^{i+j+2}/P_{k,i+j+2}+P_{i+j+2,j+1-k}\oplus W_{n}\tilde{\tilde{\omega}}_{Y}^{i+j+2}/P_{k+1,i+j+2}+P_{i+j+2,j-k},\]
which are sums of $(-1)^{i+j+1}\mathrm{proj}$ on each factor. Summing
over $k$ we get maps $\nu:W_{n}A^{ij}\to W_{n}A^{i-1j+1}$, which
induce an endomorphism $\nu$ of bidegree $(-1,1)$.

The morphism of compexes $W_{n}\tilde{\tilde{\omega}}_{Y}^{\cdot}\to W_{n}A^{\cdot0}$
given by \[
x\mapsto\frac{d\tau}{\tau}\wedge\frac{d\sigma}{\sigma}\wedge x\]
factors through $W_{n}\omega_{Y}^{\cdot}$. We get a morphism of complexes
\[
\Theta:W_{n}\omega_{Y}^{\cdot}\to W_{n}A^{\cdot}.\]

The following lemma is analogous to Theorem 9.9 of \cite{Na}. It
ensures that the resulting spectral sequence will be compatible with
the Frobenius endomorphism (defined as an endomorphism of $W_{n}$-modules).
We let $\Phi_{n}:W_{n}\omega_{Y}\to W_{n}\omega_{Y}$ be the Frobenius
endomorphism induced by the absolute Frobenius endomorphism of $(Y,M)$. 
\begin{lem}
\label{compatibility with frobenius}Let $n$ be a positive integer.
Then the following hold:
\begin{enumerate}
\item There exists a unique endomorphism $\tilde{\tilde{\Phi}}_{n}^{\cdot,\cdot}$
of $W_{n}A^{\cdot,\cdot}$ of double complexes, making the following
diagram commutative:\[
\xymatrix{W_{n+1}A^{q,m}\ar[r]^{\pi}\ar[d]_{p^{q}F} & W_{n}A^{q,m}\ar[d]^{\tilde{\tilde{\Phi}}_{n}^{q,m}}\\
W_{n}A^{q,m}\ar[r]^{\mathrm{id}} & W_{n}A^{q,m}}
.\]

\item The endomorphism $\tilde{\tilde{\Phi}}_{n}^{\cdot,\cdot}$ induces
an endomorphism $\tilde{\tilde{\Phi}}_{n}$ of the complex $W_{n}A^{\cdot}$,
fitting in a commutative diagram \[
\xymatrix{W_{n}\omega_{Y}^{\cdot}\ar[r]^{\Phi_{n}}\ar[d]_{\Theta} & W_{n}\omega_{Y}^{\cdot}\ar[d]^{\Theta}\\
W_{n}A^{\cdot}\ar[r]^{\tilde{\tilde{\Phi}}_{n}} & W_{n}A^{\cdot}}
.\]

\item Finally, the Poincare residue isomorphism $\mathrm{Res}$ fits in
the following commutative diagrams for $i,j\geq1$: \[
\xymatrix{Gr_{i,j}W_{n}\tilde{\tilde{\omega}}_{Y}^{q}\ar[r]^{\mathrm{Res}}\ar[d]_{\Psi_{n}} & W_{n}\Omega_{Y^{(i,j)}}^{q-i-j}\ar[d]^{p^{i+j}\Phi_{n}}\\
Gr_{i,j}W_{n}\tilde{\tilde{\omega}}_{Y}^{q}\ar[r]^{\mathrm{Res}} & W_{n}\Omega_{Y^{(i,j)}}^{q-i-j}}
,\]
where $\Psi_{n}$ is an endomorphism of $W_{n}\tilde{\tilde{\omega}}_{Y}^{\cdot}$
which respects the weight filtration $P_{i,j}$ and which induces
$\tilde{\tilde{\Phi}}_{n}^{\cdot,\cdot}$ on $W_{n}A^{\cdot,\cdot}$. 
\end{enumerate}
\end{lem}
\begin{proof}
The proof is essentially the same as that of Theorem 9.9 of \cite{Na}.
We emphasize only the key points. We can define a morphism $\Psi_{n}^{j,q}:W_{n}\tilde{\tilde{\omega}}_{Y}^{q}\to W_{n}\tilde{\tilde{\omega}}_{Y}^{q}$
via the composition\[
W_{n}\tilde{\tilde{\omega}}_{Y}^{q}\stackrel{\mathfrak{p}}{\to}W_{n+1}\tilde{\tilde{\omega}}_{Y}^{q}\stackrel{p^{j-1}}{\longrightarrow}W_{n+1}\tilde{\tilde{\omega}}_{Y}^{q}\stackrel{F}{\to}W_{n}\tilde{\tilde{\omega}}_{Y}^{q}.\]
The fact that these morphisms commute with the maps $\frac{d\tau}{\tau}\wedge$
and $\frac{d\sigma}{\sigma}\wedge$ follows from the proof of the
first part of Lemma \ref{compatibility with projections}. This implies
that the second diagram is commutative. The fact that the $\Psi_{n}^{\cdot,\cdot}$
respect the weight filtration follows from the analogous statement
for $\mathfrak{p}$, which is proved in Lemma \ref{compatibility with projections}
as well. This means that we can use $\Psi_{n}^{j,j+q+2}$ to define
endomorphisms $\tilde{\tilde{\Phi}}_{n}^{j,q}$ of $W_{n}A^{jq}$,
at least for $j\geq1$. For $j=0$ we use the Frobenius endomorphism
$\Phi_{n}$ of $W_{n}(\mathcal{O}_{Y^{(k+1,j-k+1)}})$ together with
the residue isomorphisms to define $\tilde{\tilde{\Phi}}_{n}^{0,q}$.
The commutativity of the first diagram now follows from the definitions,
from the commutative diagram\[
\xymatrix{W_{n+1}\tilde{\tilde{\omega}}_{Y}^{q,m}\ar[r]^{\pi}\ar[d]_{p^{q}F} & W_{n}\tilde{\tilde{\omega}}_{Y}^{q,m}\ar[d]^{\Psi_{n}^{q,m}}\\
W_{n}A^{q,m}\ar[r]^{\mathrm{id}} & W_{n}A^{q,m}}
.\]
(which is deduced from $\mathfrak{p}d=d\mathfrak{p}$ and $dF=pFd$)
and from diagram 9.2.2 of \cite{Na} in the case of a smooth morphism.
The fact that the first diagram is commutative ensures the uniqueness
of $\Phi_{n}^{q,m}$. Finally, the third commutative diagram follows
from the surjectivity of $\pi$ proved in Lemma \ref{compatibility with projections},
from from diagram 9.2.2 of \cite{Na} in the case of a smooth morphism
and from the commutative diagrams \[
\xymatrix{P_{i,j}W_{n+1}\tilde{\tilde{\omega}}_{Y}^{q}\ar[d]_{\pi}\ar[r]^{\mathrm{Res}} & W_{n+1}\Omega_{Y^{(i,j)}}^{q-i-j}\ar[d]^{\pi}\\
P_{i,j}W_{n}\tilde{\tilde{\omega}}_{Y}^{q}\ar[r]^{\mathrm{Res}} & W_{n}\Omega_{Y^{(i,j)}}^{q-i-j}}
\]
for $i,j\geq1$. \end{proof}
\begin{prop}
The sequence \[
0\to W_{n}\omega_{Y}^{\cdot}\stackrel{\Theta}{\to}W_{n}A^{\cdot0}\stackrel{d''}{\to}W_{n}A^{\cdot1}\stackrel{d''}{\to}\dots\]
is exact. \end{prop}
\begin{proof}
We follow the proof of Prop. 3.15 of \cite{Mo}. Let $\theta:W_{n}\tilde{\tilde{\omega}}_{Y}^{i-1}\oplus W_{n}\tilde{\tilde{\omega}}_{Y}^{i-1}\to W_{n}\tilde{\tilde{\omega}}_{Y}^{i}$
be defined by \[
(x,y)\mapsto\frac{d\tau}{\tau}\wedge x+\frac{d\sigma}{\sigma}\wedge y.\]
 It suffices to check that the sequence\[
W_{n}\tilde{\tilde{\omega}}_{Y}^{i-2}\stackrel{(\frac{d\sigma}{\sigma}\wedge,\frac{d\tau}{\tau}\wedge)}{\to}W_{n}\tilde{\tilde{\omega}}_{Y}^{i-1}\oplus W_{n}\tilde{\tilde{\omega}}_{Y}^{i-1}\stackrel{\theta}{\to}W_{n}\tilde{\tilde{\omega}}_{Y}^{i}\to\]
\begin{equation}
\stackrel{\frac{d\tau}{\tau}\wedge\frac{d\sigma}{\sigma}\wedge}{\longrightarrow}W_{n}\tilde{\tilde{\omega}}_{Y}^{i+2}/(P_{0,i+2}+P_{i+2,0})\stackrel{d''}{\to}W_{n}\tilde{\tilde{\omega}}_{Y}^{i+3}/(P_{1,i+3}+P_{i+3,0})\oplus W_{n}\tilde{\tilde{\omega}}_{Y}^{i+3}/(P_{0,i+3}+P_{i+3,1})\stackrel{d''}{\to}\dots\label{long exact}\end{equation}
is exact. We do this by using first a devissage by weights, reducing
to the case $n=1$ and then using the fact that the scheme $Y$ is
locally etale over a product of (the special fibers of) strictly semistable
schemes. 

We let \[
K_{-4}=W_{n}\tilde{\tilde{\omega}}_{Y}^{i-2},\]
\[
K_{-3}=W_{n}\tilde{\tilde{\omega}}_{Y}^{i-1}\oplus W_{n}\tilde{\tilde{\omega}}_{Y}^{i-1},\]
\[
K_{-2}=W_{n}\tilde{\tilde{\omega}}_{Y}^{i},\]
\[
K_{j}=\bigoplus_{k=0}^{j}W_{n}\tilde{\tilde{\omega}}_{Y}^{i+j+2}/P_{k,i+j+2}+P_{i+j+2,j-k},j\geq0.\]
For $j\geq-4,j\not=-1$ we define a double filtration of $K_{j}$
as follows:\[
P_{l,m}K_{-4}=P_{l-2,m-2}W_{n}\tilde{\tilde{\omega}}_{Y}^{i-2},\]
\[
P_{l,m}K_{-3}=P_{l-2,m-1}W_{n}\tilde{\tilde{\omega}}_{Y}^{i-1}\oplus P_{l-1,m-2}W_{n}\tilde{\tilde{\omega}}_{Y}^{i-1},\]
\[
P_{l,m}K_{-2}:=P_{l-1,m-1}W_{n}\tilde{\tilde{\omega}}_{Y}^{i},\]
\[
P_{l,m}K_{j}:=\bigoplus_{k=0}^{j}P_{l+k,m+j-k}W_{n}\tilde{\tilde{\omega}}_{Y}^{i+j+2}/P_{k,i+j+2}+P_{i+j+2,j-k},j\geq0.\]
Here we set the convention $P_{l,m}W_{n}\tilde{\tilde{\omega}}^{i}=0$
if either $l<0$ or $m<0$. The sequence (\ref{long exact}) is a
filtered sequence and to prove exactness it suffices to prove exactness
for each graded piece\[
\mbox{Gr}_{l,m}K_{j}:=P_{l,m}K_{j}/(P_{l,m-1}K_{j}+P_{l-1,m}K_{j}).\]
For $l,m\geq0$ we can rewrite the sequences of graded pieces as:\[
\mbox{Gr}_{l-2,m-2}W_{n}\tilde{\tilde{\omega}}_{Y}^{i-2}\to\mbox{Gr}_{l-2,m-1}W_{n}\tilde{\tilde{\omega}}_{Y}^{i-1}\oplus\mbox{Gr}_{l-1,m-2}W_{n}\tilde{\tilde{\omega}}_{Y}^{i-1}\to\mbox{Gr}_{l-1,m-1}W_{n}\tilde{\tilde{\omega}}_{Y}^{i}\to\]
\[
\to\mbox{Gr}_{l,m}W_{n}\tilde{\tilde{\omega}}_{Y}^{i+2}\to\mbox{Gr}_{l+1,m}W_{n}\tilde{\tilde{\omega}}_{Y}^{i+3}\oplus\mbox{Gr}_{l,m+1}W_{n}\tilde{\tilde{\omega}}_{Y}^{i+3}\to\dots.\]
For $l<0$ or $m<0$ the sequence is trivial. 

It suffices to show that the sequence of complexes\[
\mbox{Gr}_{l-2,m-2}W_{n}\tilde{\tilde{\omega}}_{Y}^{\cdot}[-2]\to\mbox{Gr}_{l-2,m-1}W_{n}\tilde{\tilde{\omega}}_{Y}^{\cdot}[-1]\oplus\mbox{Gr}_{l-1,m-2}W_{n}\tilde{\tilde{\omega}}_{Y}^{\cdot}[-1]\to\mbox{Gr}_{l-1,m-1}W_{n}\tilde{\tilde{\omega}}_{Y}^{\cdot}\stackrel{\iota}{\to}\]
\label{main exactness}\[
\mbox{Gr}_{l,m}W_{n}\tilde{\tilde{\omega}}_{Y}^{\cdot}[2]\to\mbox{Gr}_{l+1,m}W_{n}\tilde{\tilde{\omega}}_{Y}^{\cdot}[3]\oplus\mbox{Gr}_{l,m+1}W_{n}\tilde{\tilde{\omega}}_{Y}^{\cdot}[3]\to\dots\]
is exact. Note that we can check this locally. When $l,m\geq1$ we
know by Lemma \ref{identifying graded pieces} that \[
\mbox{Gr}_{l,m}W_{n}\tilde{\tilde{\omega}}_{Y}^{\cdot}\simeq W_{n}\Omega_{Y^{(l,m)}}^{\cdot}[-l-m](-l-m).\]
For $l=0$ and $m\geq1$ let $Y_{D^{0,m}}$ be the normal crossing
divisor of $D^{0,m}$ corresponding to $s=0$. In this case we have
\[
\mbox{Gr}_{l,m}W_{n}\tilde{\tilde{\omega}}_{Y}^{\cdot}\simeq[W_{n}\Omega_{D^{0,m}}(-\log Y_{D^{o,m}})\to W_{n}\Omega_{D^{0,m}}]\]
and for $l=0,m=0$ we have the quasi-isomorphism \[
\mbox{Gr}_{l,m}W_{n}\tilde{\tilde{\omega}}_{Y}^{\cdot}\simeq[W_{n}\Omega_{Z}^{\cdot}(-\log Y^{1}-\log Y^{2})\to W_{n}\Omega_{Z}^{\cdot}(-\log Y^{1})\oplus W_{n}\Omega_{Z}^{\cdot}(-\log Y^{2})\to W_{n}\Omega_{Z}^{\cdot}],\]
where $Z=\underline{Z}\otimes_{W}k.$ In any case, $\mbox{Gr}_{l,m}W_{n}\tilde{\tilde{\omega}}^{\cdot}$
satisfies the property\[
(\lim_{\leftarrow n}\mbox{Gr}_{l,m}W_{n}\tilde{\tilde{\omega}}^{\cdot})\otimes_{\mathbb{R}}^{L}\mathbb{R}_{n}\simeq\mbox{Gr}_{l,m}W_{n}\tilde{\tilde{\omega}}^{\cdot}\]
by Lemma 1.3.3 of \cite{Mo} and Lemma \ref{tensor with R_n}. By
Prop. 2.3.7 of \cite{I2}, it suffices to check exactness of the sequence
(\ref{main exactness}) for $n=1$.

For $n=1$ and working locally with our admissible lifts we know that
the exact sequence (\ref{main exactness}) is the pullback to $Y$
of the corresponding exact sequence on $Y_{1}\times_{k}Y_{2}$. We
can assume that $Y=Y_{1}\times_{k}Y_{2}$ and $Z=Z_{1}\times_{k}Z_{2}$.
Each $Y_{i}$ for $i=1,2$ is a reduced normal crossings divisor in
$Z_{i}$, for which we know that \[
\mbox{Gr}_{l_{i}-2}W_{1}\tilde{\omega}_{Y_{i}}^{\cdot}[-1]\to\mbox{Gr}_{l_{i}-1}W_{1}\tilde{\omega}_{Y_{i}}\to\]
\[
\mbox{Gr}_{l}W_{1}\tilde{\omega}_{Y_{i}}^{\cdot}[1]\to\mbox{Gr}_{l+1}W_{1}\tilde{\omega}_{Y_{i}}^{\cdot}[2]\to\dots\]
is exact, by the proof of Proposition 3.15 of \cite{Mo}. In other
words, for $i=1,2$ we have quasi-isomorphisms between the top row
and the bottom row. Multiplying the quasi-isomorphisms for $i=1$
and $2$ gives us excatly the quasi-isomorphism $\iota$ needed to
prove the exactness of (\ref{main exactness}) in the case $n=1$.
Here, we use the Cartier isomorphisms for $W_{1}\tilde{\omega}_{Y_{i}}$
and for $W_{1}\tilde{\tilde{\omega}}_{Y}$ and the fact that \[
(\omega_{(Z_{1},\tilde{N}_{1})/k}^{\cdot}\otimes_{\mathcal{O}_{Z_{1}}}\mathcal{O}_{Y_{1}})\otimes_{k}(\omega_{(Z_{2},\tilde{N}_{2})/k}^{\cdot}\otimes_{\mathcal{O}_{Z_{2}}}\mathcal{O}_{Y_{2}})\simeq\omega_{(Z,\tilde{N})}^{\cdot}\otimes_{\mathcal{O}_{Z}}\mathcal{O}_{Y},\]
where the two complexes on the left determine $W_{1}\tilde{\omega}_{Y_{i}}^{\cdot}$
for $i=1,2$ and the one on the right determines $W_{1}\tilde{\tilde{\omega}}_{Y}^{\cdot}$. \end{proof}
\begin{cor}
The morphism of complexes $\Theta:W_{n}\omega_{Y}^{\cdot}\to W_{n}A^{\cdot}$
is a quasi-isomorphism. It induces a quasi-isomorphism $\Theta:W\omega_{Y}^{\cdot}\stackrel{\sim}{\to}WA^{\cdot}$.\end{cor}
\begin{prop}
The endomorphism $\nu$ of $W_{\cdot}A^{\cdot\cdot}$ induces the
monodromy operator $N$ over $H_{\mbox{cris}}^{*}((Y,M)/(W,\mathbb{N})).$\end{prop}
\begin{proof}
We define the double complex $B_{n}^{\cdot\cdot}$ as follows:\[
B_{n}^{\cdot\cdot}=W_{n}A^{i-1j}\oplus W_{n}A^{ij},i,j\geq0\]
\[
d'(x_{1},x_{2})=(d'x_{1},d'x_{2})\]
\[
d''(x_{1},x_{2})=(d''x_{1}+\nu(x_{2}),d''x_{2}).\]
We have a morphism of complexes $\Psi:W_{n}\tilde{\omega}_{Y}^{\cdot}\to B_{n}^{\cdot}$
defined as follows, for $x\in W_{n}\tilde{\omega}_{Y}^{i}$\[
\Psi(x)=\left(\left(\frac{d\sigma}{\sigma}-\frac{d\tau}{\tau}\right)\wedge x\pmod{P_{0,i+1}+P_{i+1,0}},\frac{d\tau}{\tau}\wedge\frac{d\sigma}{\sigma}\wedge x\pmod{P_{0,i+2}+P_{i+2,0}}\right).\]
Thus we have a commutative diagram of of exact sequences of complexes:\[
\xymatrix{0\ar[r] & W_{n}\omega_{Y}^{\cdot}[-1]\ar[d]^{\Theta[-1]}\ar[r] & W_{n}\tilde{\omega}_{Y}^{\cdot}\ar[r]\ar[d]^{\Psi} & W_{n}\omega_{Y}^{\cdot}\ar[r]\ar[d]^{\Theta} & 0\\
0\ar[r] & W_{n}A^{\cdot}[-1]\ar[r] & B_{n}^{\cdot}\ar[r] & W_{n}A^{\cdot}\ar[r] & 0}
,\]
where the left and right downward arrows are quasi-isomorphisms. Thus,
$\Psi$ is also a quasi-isomorphism and the commutiatve diagram defines
an isomorphism of distinguished triangles. Thus the monodromy operator
$N$ on cohomology is induced by the couboundary operator of the bottom
exact sequence, which by construction is $\nu$. 
\end{proof}
We can compute the monodromy filtration of the nilpotent operator
$N$ on cohomology from the monodromy filtration of $\nu$ on $W_{n}A^{\cdot}$.
We will exhibit a filtration $P_{k}(W_{n}A^{\cdot})=\oplus_{i,j\geq0}P_{k}(W_{n}A^{ij})$
which satisfies the following:
\begin{enumerate}
\item $\nu(P_{k}(W_{\cdot}A^{\cdot}))\subset P_{k-2}(W_{\cdot}A^{\cdot})(-1)$
\item For $k\geq0$ the induced map $\nu^{k}:\mbox{Gr}_{k}(W_{\cdot}A^{\cdot})\to\mbox{Gr}_{-k}(W_{\cdot}A^{\cdot})(-k)$
is an isomorphism. 
\end{enumerate}
A filtration satisfying these two properties must be the monodromy
filtration of $\nu$. 
\begin{note}
From now on, we will not work in the category $\mathfrak{C}$ of complexes
of sheaves of $W$-modules but rather in $\mathbb{Q}\otimes\mathfrak{C}$,
which is the category with the same set of objects as $\mathfrak{C}$,
but with morphisms $\mathbb{Q}\otimes\mbox{Hom}_{\mathfrak{C}}(A,B)$.
We will in fact identify the monodromy filtration of $\nu$ on $\mathbb{Q}\otimes W_{n}A^{\cdot}$,
but for simplicity of notation we still denote an object $A$ of $\mathfrak{C}$
as $A$ when we regard it as an object of $\mathbb{Q}\otimes\mathfrak{C}$. 
\end{note}
Define $P_{l}(W_{n}A^{\cdot\cdot}):=\oplus_{i,j\geq0}P_{l}(W_{n}A^{ij})$
for $l\geq0$, where \[
P_{l}(W_{n}A^{ij}):=\begin{cases}
0 & \mbox{ if }l<2n-2-j\\
\oplus_{k=0}^{j}(\sum_{m=0}^{l-2n+2+j}P_{k+m+1,2j-k+l-2n-m+3}W_{n}\tilde{\tilde{\omega}}_{Y}^{i+j+2}/P_{k,i+j+2}+P_{j-k,i+j+2}) & \mbox{ if }l\geq2n-2-j\end{cases}.\]
It is easy to check that $\nu(P_{l}(W_{n}A^{ij}))\subset P_{l-2}W_{n}A^{i+1,j-1}$.
Moreover, we can also compute the graded pieces $\mbox{Gr}_{l}(W_{n}A^{\cdot\cdot})=\bigoplus_{i,j\geq0}\mbox{Gr}_{l}(W_{n}A^{ij}),$
where \[
\mbox{Gr}_{l}(W_{n}A^{ij})==\begin{cases}
0 & \mbox{ if }l<2n-2-j\\
\oplus_{k=0}^{j}\oplus_{m=0}^{l-2n+2+j}\mbox{Gr}_{k+m+1,2j-k+l-2n-m+3}W_{n}\tilde{\tilde{\omega}}_{Y}^{i+j+2} & \mbox{ if }l\geq2n-2-j\end{cases}.\]
For $l=2n-2+h,$ with $h>0$ we claim that $\nu$ induces an injection
$\mbox{Gr}_{l}(W_{n}A^{ij})\hookrightarrow\mbox{Gr}_{l-1}(W_{n}A^{ij})$.
This can be verified through a standard combinatorial argument. We
have \[
\mbox{Gr}_{l}(W_{n}A^{ij})=\bigoplus_{k=0}^{j}\bigoplus_{m=0}^{h+j}\mbox{Gr}_{(k+m)+1,2j+h+1-(k+m)}W_{n}\tilde{\tilde{\omega}}_{Y}^{i+j+2}\]
and \[
\mbox{Gr}_{l-1}(W_{n}A^{ij})=\bigoplus_{k=0}^{j+1}\bigoplus_{m=0}^{h+j-1}\mbox{Gr}_{(k+m)+1,2j+h+1-(k+m)}W_{n}\tilde{\tilde{\omega}}_{Y}^{i+j+2}.\]
The map $\nu$ sends the term corresponding to a pair $(k,m)$ to
the direct sum of terms corresponding to $(k,m)$ and to $(k+1,m-1)$.
Therefore, it is easy to see that $\nu$ restricted to the direct
sum of terms for which $k+m$ is constant is injective, so $\nu$
is injective. Moreover, we see that $\nu^{h}$ induces an isomorphism
$\mbox{Gr}_{2n-2+h}(W_{n}A^{ij})\simeq\mbox{Gr}_{2n-2-h}(W_{n}A^{i-h,j+h})$,
since the terms on the right hand side are of the form \[
\bigoplus_{k=0}^{j}\bigoplus_{m=0}^{h+j}\mbox{Gr}_{(k+m)+1,2j+h+1-(k+m)}W_{n}\tilde{\tilde{\omega}}_{Y}^{i+j+2}\]
and the terms on the left hand side are of the form\[
\bigoplus_{m=0}^{j}\bigoplus_{k=0}^{h+j}\mbox{Gr}_{(k+m)+1,2j+h+1-(k+m)}W_{n}\tilde{\tilde{\omega}}_{Y}^{i+j+2},\]
so on either side we have the same number of terms corresponding to
$k+m$. Since the filtration $P_{l}(W_{n}A^{\cdot\cdot})$ satisfies
the two properties above, it must be the monodromy filtration of $\nu$. 

Note that the differentials $d''$ on $\mbox{Gr}_{l}(W_{\cdot}A^{\cdot\cdot})$
are always $0$. Using the isomorphisms in Corollary \ref{identifying graded pieces}
we can rewrite \[
\mbox{Gr}_{2n-2+h}(W_{\cdot}A^{\cdot})\simeq\bigoplus_{j\geq0,j\geq-h}\bigoplus_{k=0}^{j}\bigoplus_{m=0}^{j+h}(W_{\cdot}\Omega_{Y^{(k+m+1,2j+h+1-(k+m)}}^{\cdot})[-2j-h](-j-h).\]
Thus, we get the following theorem.
\begin{thm}
\label{weight spectral sequence}There is a spectral sequence\[
E_{1}^{-h,i+h}=\bigoplus_{j\geq0,j\geq-h}\bigoplus_{k=0}^{j}\bigoplus_{m=0}^{j+h}H_{\mbox{cris}}^{i-2j-h}(Y^{(k+m+1,2j+h+1-(k+m)}/W)(-j-h)\]
\[
\Rightarrow H_{\mbox{cris}}^{i}(Y/W).\]
\end{thm}
\begin{rem}
Note that the closed strata $Y^{(l_{1},l_{2})}$ are proper and smooth
so the $E^{-h,i+h}$ terms of the spectral sequence are strictly pure
of weight $i+h$. If the above spectral sequence degenerates at the
first page, then $H_{\mbox{cris}}^{i}(Y/W)$ is pure of weight $i$. 
\end{rem}

\section{Proof of the main theorem}

In this section we prove the main theorem. By the discussion at the
end of Section 2 its proof reduces to the following proposition. 
\begin{prop}
Let $\mathcal{A}_{U_{\mathrm{Iw}}}^{m_{\xi}}$ be the universal abelian
variety over $X_{U_{\mathrm{Iw}}}^{m_{\xi}}$. The direct limit of
log crystalline cohomologies \[
\lim_{\substack{\to\\
U_{\mathrm{Iw}}}
}a_{\xi}(H_{\mathrm{cris}}^{2n-2+m_{\xi}}(\mathcal{A}_{U_{\mathrm{Iw}}}^{m_{\xi}}\times_{\mathcal{O}_{K}}k/W)\otimes_{W}\bar{\mathbb{Q}}_{l}(t_{\xi}))[\Pi^{1,\mathfrak{S}}]\]
is pure of a certain weight. \end{prop}
\begin{proof}
Recall that we've chosen \[
U_{\mathrm{Iw}}=U^{l}\times U_{l}^{\mathfrak{p}_{1},\mathfrak{p}_{2}}(m)\times\mathrm{Iw}_{n,\mathfrak{p}_{1}}\times\mathrm{Iw}_{n,\mathfrak{p}_{2}}\subset G(\mathbb{A}^{\infty}).\]
Pick $m$ large enough such that $(\pi_{l})^{U_{l}^{\mathfrak{p}_{1},\mathfrak{p}_{2}}(m)\times\mathrm{Iw}_{n,\mathfrak{p}_{1}}\times\mathrm{Iw}_{n,\mathfrak{p}_{2}}}\not=0$,
where $\pi_{l}\in\mathrm{Irr}_{l}(G(\mathbb{Q}_{l})$ is such that
$BC(\pi_{l})=\iota_{l}^{-1}\Pi_{l}$. The results of Sections 3 and
4 apply to $\mathcal{A}_{U_{\mathrm{Iw}}}^{m_{\xi}}$. We have a stratification
of its special fiber by closed Newton polygon strata $\mathcal{A}_{U_{\mathrm{Iw}},S,T}^{m_{\xi}}$
with $S,T\subseteq\{1,\dots,n\}$ non-empty. By Theorem \ref{weight spectral sequence}
we have a spectral sequence

\[
E_{1}^{-h,i+h}=\bigoplus_{j\geq0,j\geq-h}\bigoplus_{k=0}^{j}\bigoplus_{m=0}^{j+h}\bigoplus_{\substack{\#S=k+m+1\\
\#T=2j+h+1-(k+m)}
}H_{\mathrm{cris}}^{i-2j-h}((\mathcal{A}_{U_{\mathrm{Iw}},S,T}^{m_{\xi,}}/W)(-j-h)\]
\[
\Rightarrow H_{\mathrm{cris}}^{i}(\mathcal{A}_{U_{\mathrm{Iw}}}^{m_{\xi}}\times_{\mathcal{O}_{K}}k/W).\]
We replace the cohomology degree $i$ by $i+m_{\xi}$, tensor with
$\bar{\mathbb{Q}}_{l}(t_{\xi})$, apply $a_{\xi}$ (which is obtained
from a linear combination of etale morphisms), passing to a direct
limit over $U^{l}$ and taking the $\Pi^{1,\mathfrak{S}}$-isotypic
components we get a spectral sequence:\[
E_{1}^{-h,i+h}=\bigoplus_{j\geq0,j\geq-h}\bigoplus_{k=0}^{j}\bigoplus_{m=0}^{j+h}\bigoplus_{\substack{\#S=k+m+1\\
\#T=2j+h+1-(k+m)}
}\lim_{\substack{\to\\
U^{l}}
}(a_{\xi}H_{\mathrm{cris}}^{i+m_{\xi}-2j-h}((\mathcal{A}_{U_{\mathrm{Iw}},S,T}^{m_{\xi,}}/W)(-j-h)\otimes_{W,\tau_{0}}\bar{\mathbb{Q}}_{l}(t_{\xi}))[\Pi^{1,\mathfrak{S}}]\]
\[
\Rightarrow\lim_{\substack{\to\\
U^{l}}
}(a_{\xi}H_{\mathrm{cris}}^{i+m_{\xi}}(\mathcal{A}_{U_{\mathrm{Iw}}}^{m_{\xi}}\times_{\mathcal{O}_{K}}k/W)\otimes_{W,\tau_{0}}\bar{\mathbb{Q}}_{l}(t_{\xi}))[\Pi^{1,\mathfrak{S}}].\]
For any compact open subgroup $U^{l}\subset G(\mathbb{A}^{\infty,l})$
and any prime $p\not=l$ with isomorphism $\iota_{p}:\mathbb{\bar{Q}}_{p}\stackrel{\sim}{\to}\mathbb{C}$
set $\xi':=(\iota_{p})^{-1}\iota_{l}\xi$ and $\Pi':=(\iota_{p})^{-1}\Pi^{1}$. 

We have \[
\dim_{\bar{\mathbb{Q}}_{l}}(\lim_{\substack{\to\\
U^{l}}
}a_{\xi}H_{\mathrm{cris}}^{i+m_{\xi}-2j-h}((\mathcal{A}_{U_{\mathrm{Iw}},S,T}^{m_{\xi,}}/W)(-j-h)\otimes_{W,\tau_{0}}\bar{\mathbb{Q}}_{l}))[\Pi^{1,\mathfrak{S}}]^{U^{l}}\]
\[
=\dim_{\mathbb{\bar{Q}}_{p}}(\lim_{\substack{\to\\
U^{l}}
}a_{\xi'}H^{i+m_{\xi'}-2j-h}(\mathcal{A}_{U_{\mathrm{Iw}},S,T}^{m_{\xi}},\bar{\mathbb{Q}}_{p}))[(\Pi')^{\mathfrak{S}}]^{U^{l}}\]
\[
=\dim_{\bar{\mathbb{Q}}_{p}}(\lim_{\substack{\to\\
U^{l}}
}H^{i-2j-h}(X_{U_{\mathrm{Iw}},S,T},\mathcal{L}_{\xi'}))[(\Pi')^{\mathfrak{S}}]^{U^{l}}.\]
The first equality is a consequence of the main theorem of \cite{GM}
and of Theorem 2 (2) of \cite{KM}. The former proves that crystalline
cohomology is a Weil cohomology theory in the strong sense. The latter
is the statement that the characteristic polynomial on $H^{i}(X)$
of an integrally algebraic cycle on $X\times X$ of codimension $n$,
for a projective smooth variety $X/k$ of dimension $n$, is independent
of the Weil cohomology theory $H$. 

The dimension in the third row is equal to $0$ unless $i=2n-2$ by
Prop. 5.10 of \cite{C}. Therefore, $E_{1}^{-h,i+h}=0$ unless $i=2n-2$,
so the $E_{1}$ page of the spectral sequence is concentrated on a
diagonal. The spectral sequence degenerates at the $E_{1}$ page and
the term corresponding to $E_{1}^{h,2n-2+h}$ is strictly pure of
weight $h+2n-2+m_{\xi}-2t_{\xi}$, which shows that the abutment is
pure. \end{proof}


\begin{thebibliography}{Ts}
\bibitem[BLGGT1]{BLGGT1}Barnet-Lamb, T., Gee, T., Geraghty, D. and
Taylor, R: Local-global compatibility for $l=p$, I, to appear in
Ann. Fac. Sci. Toulouse Math. 

\bibitem[BLGGT2]{BLGGT2}Barnet-Lamb, T., Gee, T., Geraghty, D. and
Taylor, R: Local-global compatibility for $l=p$, II, preprint (2011). 

\bibitem[C]{C}Caraiani, A.: Local-global compatibility and the action
of monodromy on nearby cycles, to appear in Duke Math. J., arXiv:1010.2188.

\bibitem[CH]{CH}Chenevier, G. and Harris, M.: Construction of automorphic
Galois representations, II, http://people.math.jusssieu.fr/\textasciitilde{}harris/ConstructionII.pdf 

\bibitem[D]{D}Deligne, P: Equations differentielles a points singuliers
reguliers, Lecture Notes in Math. 163, Springer-Verlag, Belin, 1970. 

\bibitem[DI]{DI}Deligne, P. and Illusie, L.: Relevements modulo $p^{2}$
et decomposition du complexe de de Rham, Invent. Math. 89 (1987),
247-270. 

\bibitem[GM]{GM}Gillet, H. and Messing, W.: Cycle classes and Riemann-Roch
for crystalline cohomology, Duke Math. J. 1987, vol. 55 (3), 501-538. 

\bibitem[EGA IV]{EGA IV}Grothendieck, A and Diedonne, J..: Elements
de geometrie algebrique IV. Etude locale des schemas et des morphismes
de schemas, Quatrieme partie, Publ. Math. I.H.E.S. 32 (1967), 5-361. 

\bibitem[HT]{HT}Harris, M. and Taylor, R.: The geometry and cohomology
of some simple Shimura varieties, Princeton University Presss, no.
151, Annals of Math. Studies, Princeton, New Jersey (2001).

\bibitem[H]{H}Hyodo, O.: On the de Rham-Witt complex attached to
a semistable family, Comp. Math. 1991, vol. 78 (3), 241-260. 

\bibitem[HK]{HK}Hyodo, O. and Kato, K.: Semistable reduction and
crystalline cohomology with logarithmic poles, Asterisque 223, 1994,
221-268. 

\bibitem[I]{I}Illusie, L.: Complexe de de Rham-Witt et cohomologie
cristalline, Ann. scien. de l'E.N.S. , 4, 1979, vol. 12 (4), 501-661. 

\bibitem[I2]{I2}Illusie, L: Finiteness, duality and Kunneth theorems
in the cohomology of the de Rhan-Witt complex, in Algebraic Geometry,
Lecture Notes in Math. 1016, Springer-Verlag, Berlin, 1983, 20-72. 

\bibitem[IR]{IR}Illusie, L. and Raynaud, M.: Les suites spectrales
associees au complexe de de Rham Witt, Publ. Math. I.H.E.S., 1983,
vol 57, 73-212. 

\bibitem[K1]{K}Kato, K: Logarithmic structures of Fontaine-Illusie,
Algebraic analysis, geometry and number theory, Baltimore, MD, 1988,
191-224.

\bibitem[K2]{K2}Kato, K.: Toric singularities, American Journal of
Mathematics 116 (5), 1999, 1073-1099.

\bibitem[KM]{KM}Katz, N. and Messing, W.: Some consequences of the
Riemann hypothesis for varieties over finite fields, Invent. Math.,
1974, vol. 23, 73-77.

\bibitem[Man]{Man}Mantovan, E.: On the cohomology of certain PEL-type
Shimura varieties, Duke Math. J. 129 (2005), 573-610. 

\bibitem[Mo]{Mo}Mokrane, A.: La suite spectrale de poids en cohomologie
de Hyodo-Kato, Duke Math. J. 1993, vol. 72 (2), 301-337. 

\bibitem[Na]{Na}Nakkajima, Y.: p-Adic weight spectral sequences of
Log varieties, J. Math. Sci. Univ. Tokyo 2005, vol. 12, 513-661. 

\bibitem[N1]{N}Niziol, W.: Semistable conjecture via $K$-theory,
Duke Math. J. 2008, vol 141 (1), 151-178.

\bibitem[N2]{N2}Niziol, W.: Toric singularities: log-blow-ups and
global resolutions, J. Alg. Geom. 15 (1), 2006, 1-29. 

\bibitem[Sh]{Sh}Shin, S.W.: Galois representations arising from some
compact Shimura varieties, preprint (2009), to appear in Ann. of Math. 

\bibitem[TY]{TY}Taylor, R. and Yoshida, T.: Compatibility of local
and global Langlands correspondences, J. Amer. Math. Soc. 20 (2) (2007),
467-493. 

\bibitem[Ts]{Ts}Tsuji, T.: Saturated morphisms of logarithmic schemes,
preprint. 
\end{thebibliography}
\end{document}